\documentclass[12pt,a4paper]{article}

\usepackage[font=small,margin=2cm]{caption}

\usepackage{authblk}
\usepackage[margin=1.8cm,bottom=2cm]{geometry}
\usepackage{t1enc}
\usepackage[utf8]{inputenc}
\usepackage{amsmath,amsthm,amssymb}
\usepackage{graphicx}
\usepackage{enumerate}
\usepackage{hyperref}
\usepackage{bm}
\usepackage{comment}
\usepackage{amsfonts}
\usepackage{graphicx,caption}
\usepackage{bm}
\usepackage{amsmath, amsthm, amssymb}
\usepackage{graphicx}
\usepackage{hyperref}
\usepackage{relsize}
\usepackage{blkarray}
\usepackage{tabstackengine}
\stackMath

\usepackage[table]{xcolor}

\usepackage{algpseudocode}

\usepackage{bbm}

\theoremstyle{plain}
\usepackage{amsthm}
\makeatletter
\newcommand{\newreptheorem}[2]{\newtheorem*{rep@#1}{\rep@title}\newenvironment{rep#1}[1]{\def\rep@title{#2 \ref*{##1}}\begin{rep@#1}}{\end{rep@#1}}}
\makeatother

\newtheorem{theorem}{Theorem}[section]
\newtheorem*{theorem-non}{Theorem}
\newtheorem*{non-lemma}{Lemma}
\newtheorem{lemma}[theorem]{Lemma}
\newreptheorem{lemma}{Lemma}

\newtheorem{corollary}[theorem]{Corollary}

\newtheorem{conjecture}[theorem]{Conjecture}

\theoremstyle{definition}
\newtheorem{remark}[theorem]{Remark}

\DeclareMathOperator{\range}{Im }

\DeclareMathOperator{\rang}{rank}

\DeclareMathOperator{\cok}{cok}

\DeclareMathOperator{\supp}{supp}

\DeclareMathOperator{\Aut}{Aut}

\DeclareMathOperator{\Var}{Var}

\DeclareMathOperator{\dTV}{d_{TV}}

\begin{document}

\title{The rank evolution of block bidiagonal matrices \\over finite fields}
\author{Andr\'as M\'esz\'aros\thanks{{\tt meszaros@renyi.hu}\\The author was supported by the NKKP-STARTING 150955 project and the Marie Sk\l{}odowska-Curie Postdoctoral Fellowship "RaCoCoLe".}}
\date{}
\affil{HUN-REN Alfr\'ed R\'enyi Institute of Mathematics,\\ Budapest, Hungary}
\maketitle
\begin{abstract}
We investigate uniform random block lower bidiagonal matrices over the finite field~$\mathbb{F}_q$, and prove that their corank undergoes a phase transition.

First, we consider block lower bidiagonal matrices with $(k_n+1)\times k_n$ blocks where each block is of size $n\times n$. We prove that if $k_n\ll q^{n/2}$, then the corank of these matrices is $0$ with high probability, and if $k_n\gg q^{n/2}$, then the corank has Gaussian fluctuations.

Second, we consider block lower bidiagonal matrices with $k_n\times k_n$ blocks where each block is of size $n\times n$. We prove that if $k_n\ll q^{n/2}$, then the corank exhibits the same constant order fluctuations as the corank of the matrix products considered by Nguyen and Van Peski, and if $k_n\gg q^{n/2}$, then the corank has Gaussian fluctuations.

Finally, we also consider a truncated version of the first model, where we prove that at $k_n\approx q^{n/2}$, we have a phase transition between a Cohen-Lenstra and a Gaussian limiting behavior of the corank. We also show that there is a localization/delocalization phase transition for the vectors in the kernels of these matrices at the same critical point.

In all three cases, we also provide a precise description of the behavior of the corank at criticality.

\end{abstract}

\section{Introduction}

Let $A_n$ be a uniform random $n\times n$ matrix over $\mathbb{F}_p$. The corank of $A_n$, that is, $\dim\ker A_n$ has a limiting distribution. More specifically, for any non-negative integer $k$, we have
\begin{equation}\label{universaleqint}\lim_{n\to\infty}\mathbb{P}(\dim\ker A_n=k)=p^{-k^2} \prod_{i=1}^k (1-p^{-i})^{-2}\prod_{i=1}^\infty (1-p^{-i}).\end{equation}
It turns out that this limiting distribution, which is often referred to as the Cohen-Lenstra distribution, is universal: For a large class of random matrices over $\mathbb{F}_p$, the corank has the same limiting distribution as in~\eqref{universaleqint}. Understanding the extent of the universality of the Cohen-Lenstra limiting distribution is a central question in $p$-adic random matrix theory. Very often, universality can be also established for a finer invariant called the cokernel, see Section~\ref{extandedintro} for more details.   
\newpage 
Uniform random band matrices provide one family of random matrices over $\mathbb{F}_p$, where we have a precise characterization of the cases where the Cohen-Lenstra limiting distribution occurs. The author~\cite{meszaros2024Zpband} proved that if $B_n$ is a uniform random $n\times n$ matrix over~$\mathbb{F}_p$ with band width $w_n$, then $\dim\ker (B_n)$ has Cohen-Lenstra limiting distribution if and only if $\lim_{n\to\infty} np^{-w_n}=0$. These results were also motivated by the metal/insulator phase transition of Gaussian band matrices, see Section~\ref{extandedintro} for more details.

The author conjectured that at criticality a new one-parameter family of limiting distributions arises:
\begin{conjecture}\label{conjecturecriticalInt}
There is a one-parameter family of distributions $(\mu_{t})_{t\in\mathbb{R}_+}$ on the set of nonnegative integers with the following property. Let $n_1<n_2<\dots$ be a sequence of positive integers. Let ${B}_i$ be a uniform random  $n_i\times n_i$ band matrix over $\mathbb{F}_p$ with band width $w_i$. Let us assume that $\lim_{i\to\infty} p^{-w_i} n_i=t$. Then 
\[\lim_{i\to\infty} \mathbb{P}(\dim\ker{B}_i=k)=\mu_t(k).\]
\end{conjecture}
This conjecture is still open. 

In Theorem~\ref{theoremtruncated} below, we consider truncated block bidiagonal matrices $\widehat{C}_{2k_n}^{(n)}$, which serve as a simplified model of a band matrix over $\mathbb{F}_p$ with band width $n/2$.  We show that the corank of these matrices exhibits a  Cohen-Lenstra/non-Cohen-Lenstra phase transition similarly to band matrices. Moreover, by part~\eqref{theoremtruncatedpart2} of Theorem~\ref{theoremtruncated}, we also have a precise understanding of the behavior of the corank at criticality. Thus, for $\widehat{C}_{2k_n}^{(n)}$, we can prove the analogue of Conjecture~\ref{conjecturecriticalInt}.

Besides the truncated block bidiagonal matrices mentioned above, we also consider two other closely related models, see Section~\ref{SecStatement} for the definitions of these models. The table below gives an overview of our main results.

\[
\begin{array}{c|c|c|c}
\text{matrix} & \text{size} & \text{main result}&\text{connections to}\\
\hline
C_{2k} & (k+1)n\times kn & \text{Theorem~\ref{thmdimkerC2k}}&\\
C_{2k-1} & kn\times kn & \text{Theorem~\ref{thmdimkerC2k-1}}& \text{matrix products}\\
\widehat C_{2k} & (k+1)n\times (k+1)n & \text{Theorems~\ref{theoremtruncated} and \ref{thmlocdeloc}}&\text{band matrices}
\end{array}
\]

Although the differences in the construction of these models may seem small, they have a significant effect on the behavior of the corank. Based on the relationship of the number of blocks and the sizes of the blocks, we will identify a phase transition for all three models. We also provide a precise description of the behavior of the corank at criticality.

Another motivation for studying the corank of block bidiagonal matrices comes from the recent results of Nguyen and Van Peski regarding the distribution of the coranks of matrix products over~$\mathbb{F}_p$~\cite{nguyen2024rank,nguyen2024universality}. A connection between matrix products and block bidiagonal matrices was observed earlier by the author~\cite{meszaros2024universal}, see also Section~\ref{secproductvsbidiagonal}. Theorem~\ref{thmdimkerC2k-1} below shows that for a certain range of parameters the corank of block bidiagonal matrices shows the same limiting distribution as the corank of matrix products, but after a while the behavior of these two random matrix models deviates from each other.

Our results fit into the larger program of considering the corank/cokernel of random matrix models with some underlying geometric structure, and trying to understand how this underlying structure affects the distribution of the corank/cokernel.

\subsection{The statement of the results}\label{SecStatement}

In the previous section, we discussed random matrices over $\mathbb{F}_p$, where $p$ is a prime. But our results are true for any finite field. We only restricted our attention to primes to keep our discussion consistent with Section~\ref{extandedintro}, where we consider the more general context of cokernels. In this context, the mod $p$ corank comes up as the rank of the Sylow $p$-subgroup of the cokernel, which only makes sense if $p$ is a prime.

Thus, let $q$ be a prime power. Let $A=\left({A}_{i,j}\right)_{i,j=1}^\infty$ be a random infinite block matrix over the finite field $\mathbb{F}_q$ such that all the blocks are $n\times n$ matrices, and $A$ is a block lower bidiagonal matrix, that is, $A_{i,j}=0$ whenever $j \notin \{i,i-1\}$. Moreover, the other blocks $(A_{i,j})_{j\in \{i,i-1\}}$ are chosen independently uniformly at random from the set of all $n\times n$ matrices over~$\mathbb{F}_q$. 

For $k\ge 1$, let $C_{2k-1}$ be the submatrix of $A$ determined by the first $kn$ rows and $kn$ columns, that is, 

\[
\setstackgap{L}{2\baselineskip}
\fixTABwidth{T}
C_{2k-1}=\parenMatrixstack{
{A}_{1,1}&{0}&\ddots&&&&\\
{A_{2,1}}&{A}_{2,2}&{0}&\ddots&&&\\
{0}&{A_{3,2}}&{A}_{3,3}&{0}&\ddots&&\\
\ddots&{0}&{A_{4,3}}&{A}_{4,4}&{0}&\ddots&&\\
&\ddots&\ddots&\ddots&\ddots&\ddots&\ddots\\
& &\ddots&{0}&\,{A_{k-1,k-2}}\,&{A}_{k-1,k-1}&0\\
& &&\ddots&{0}&{A_{k,k-1}}&{A}_{k,k}
}.\]

Furthermore, let $C_{2k}$ be the submatrix of $A$ determined by the first $(k+1)n$ rows and $kn$ columns, that is,

\[
\setstackgap{L}{2\baselineskip}
\fixTABwidth{T}
C_{2k}=\parenMatrixstack{
{A}_{1,1}&{0}&\ddots&&&&\\
{A_{2,1}}&{A}_{2,2}&{0}&\ddots&&&\\
{0}&{A_{3,2}}&{A}_{3,3}&{0}&\ddots&&\\
\ddots&{0}&{A_{4,3}}&{A}_{4,4}&{0}&\ddots&&\\
&\ddots&\ddots&\ddots&\ddots&\ddots&\ddots\\
& &\ddots&{0}&\,{A_{k-1,k-2}}\,&{A}_{k-1,k-1}&0\\
& &&\ddots&{0}&{A_{k,k-1}}&{A}_{k,k}\\
& &&&\ddots&{0}&{A_{k+1,k}}
}.\]

When we want to emphasize the dependence on $n$, we write $C_i^{(n)}$ in place of $C_i$.

In the rest of the Introduction, we assume that $n=2m$ is even. See Section~\ref{oddn} for the necessary modifications in the case of an odd $n$.

We understand the evolution of $\dim\ker C_i$, by analyzing the rank increments defined by
\[X_i=\rang(C_i)-\rang(C_{i-1}).\]

We describe the limiting behavior of the process $(X_i)$ in terms of a continuous time random walk $Z_t$ on $\mathbb{Z}$ with transition rate matrix

\[Q_Z(a,b)=\begin{cases}
q^{-a}&\text{if }b=a+1,\\
q^a&\text{if }b=a-1,\\
-q^a-q^{-a}&\text{if }$a=b$,\\
0&\text{otherwise.}
\end{cases}\]

The initial distribution of $Z_0$ is to be specified later. If $Z_0=a$ for some deterministic constant~$a$, then sometimes we will write $Z_{a,t}$ in place of $Z_t$.

Note that $Z_t$ is more likely to move towards $0$ than away from $0$, and this bias is increasing with the distance from $0$.

The continuous time random walk $Z_t$ will enter the picture as the limit of the rescaled and recentered version $(Y_t^{(n)})_{t\in \mathbb{R}_{\ge 0}}$ of $X_i$ defined as
\[Y_t^{(n)}=X(2\lfloor q^{n/2}(q-1) t\rfloor)-n/2,\]
see Theorem~\ref{theoremcoupling} for precise statements. Note that after the statement of Theorem~\ref{theoremcoupling}, we give a quick explanation why it is true. Unsurprisingly, the scaling factor $q^{n/2}(q-1)$ appears in our theorems as the point of phase transition for the coranks.

For $a\in \mathbb{Z}$ and $t>0$, assuming the initial condition $Z_0=a$, let $D_{a,t}$ be the number of downward jumps made by $Z$ until time $t$ (including a possible jump at time $t$ provided that it is a downward jump).

For any fixed $t>0$, let the random variable $(D_{-\infty,t},Z_{-\infty,t})$ be obtained as the limit $(D_{a,t},Z_{a,t})$ as $a$ tends to $-\infty$. The existence of such a random variable is provided by Lemma~\ref{Dinftyexists} below. Our need to understand $Z_t$ started from $-\infty$ comes from the fact the process $Y^{(n)}_t$ has initial condition $Y^{(n)}_0=-\frac{n}2$, which tends to $-\infty$.

Now we are ready to state our main results describing the limiting distribution of $\dim\ker C_i$. It turns out that this behavior depends on the parity of $i$. Theorem~\ref{thmdimkerC2k} and Theorem~\ref{thmdimkerC2k-1} describe the even and odd cases, respectively.

\begin{theorem}\label{thmdimkerC2k}

Let $k_n$ be a sequence of positive integers, and let $t_n=\frac{q^{-n/2}}{q-1}k_n$. 
 \begin{enumerate}[(a)]
\item\label{thmdimkerC2kpart1} Assume that $\lim_{n\to\infty}t_n=0$, then $\dim\ker {C}^{(n)}_{2k_n}$ converges to $0$ in probability.
\item\label{thmdimkerC2kpart2} Assume that $\lim_{n\to\infty}t_n=t$, where $0<t<\infty$, then $\dim\ker {C}^{(n)}_{2k_n}$ converges to $D_{-\infty,t}$ in total variation distance.

\item\label{thmdimkerC2kpart3} 
Let 
\begin{equation}\label{mundef}
\mu_n=(q-1)q^{n/2}\left(n-2\sum_{h=0}^{n}\pi_n(h)h\right),
\end{equation}
where $\pi_n(h)$ is defined in Lemma~\ref{lemmaMarkov} below.

There is a constant $\sigma>0$ such that if $\lim_{n\to\infty}t_n=\infty$, then
\[\frac{\dim\ker {C}^{(n)}_{2k_n}-\mu_n t_n}{\sigma \sqrt{t_n}}\]
converges in distribution to a standard normal variable.

Moreover,
\begin{equation}\label{munas}\mu_n=\frac{2 \sum_{i=1}^{\infty} q^{-i(i-1)}}{1+2\sum_{i=1}^\infty q^{-i^2}}+O(q^{-n/2}).\end{equation}
\end{enumerate}
\end{theorem}
Note that although we could merge the cases \eqref{thmdimkerC2kpart1} and \eqref{thmdimkerC2kpart2} of Theorem~\ref{thmdimkerC2k} together by observing that $D_{-\infty,0}=0$, we decided to keep the subcritical and critical cases separate.

\begin{theorem}\label{thmdimkerC2k-1}

Let $k_n$ be a sequence of positive integers, and let $t_n=\frac{q^{-n/2}}{q-1}k_n$. 
\begin{enumerate}[(a)]

\item\label{thmdimkerC2k-1part1} Assume that $\lim_{n\to\infty}t_n=0$, then the total variation distance of $\dim\ker {C}^{(n)}_{2k_n-1}$ and \break $\dim\ker A_1^{(n)}A_2^{(n)}\cdots A_{k_n}^{(n)}$ converges to $0$, where $A_1^{(n)},A_2^{(n)},\dots, A_{k_n}^{(n)}$ are i.i.d. uniform random $n\times n$ matrices over $\mathbb{F}_q$. 

\item\label{thmdimkerC2k-1part2} Assume that $\lim_{n\to\infty}t_n=t$, where $0<t<\infty$, then \[\dim\ker {C}^{(n)}_{2k_n+1}-\frac{n}2 \text{ converges to } D_{-\infty,t}+Z_{-\infty,t}\text{ in total variation distance.}\]

\item\label{thmdimkerC2k-1part3} Assume that $\lim_{n\to\infty}t_n=\infty$, then
\[\frac{\dim\ker {C}^{(n)}_{2k_n+1}-\frac{n}2-\mu_n t_n}{\sigma \sqrt{t_n}}\]
converges in distribution to a standard normal variable. Here the constants $\mu_1,\mu_2,\dots$ and $\sigma$ are the same as in Theorem~\ref{thmdimkerC2k}.

\end{enumerate}
\end{theorem}

See also Remark~\ref{remarkdeviateafterwhile}, where we explain that in case of $\lim_{n\to\infty} t_n>0$, the behavior of $\dim\ker {C}^{(n)}_{2k_n-1}$ and  $\dim\ker A_1^{(n)}A_2^{(n)}\cdots A_{k_n}^{(n)}$ are different.

 Combining these results on the corank of matrix products by Nguyen and Van Peski (more specifically, \cite[Theorem 1.4.]{nguyen2024universality} and \cite[Theorem 10.1]{van2023local}) with part \eqref{thmdimkerC2k-1part1} of Theorem~\ref{thmdimkerC2k-1}, we obtain the following theorem. 

\begin{theorem}\label{thmmatrixproduct}
Assume that $q$ is a prime.\footnote{We believe that this theorem is also true for any prime power. However, the results of \cite{nguyen2024universality} and \cite{van2023local} are only stated in the prime case. Although we think that the arguments can be extended to prime powers, we do not wish to spend time on this issue in this paper.}
\begin{enumerate}[(a)]
 \item Assume that $k_n=k$ is constant, then for all $j\ge 0$, we have
\[\lim_{n\to\infty}
\mathbb{P}\left(\dim\ker C^{(n)}_{2k_n-1}=j\right)=\sum_{\substack{r_1,r_2,\dots,r_k\ge 0\\r_1+r_2+\cdots+r_k=j}}(q^{-1};q^{-1})_{\infty}^k \prod_{i=1}^k \frac{q^{-r_i(r_1+\dots+r_i)}}{(q^{-1};q^{-1})_{r_i}},\]
where we used the $q$-Pochhammer symbol $(a;q^{-1})_r=\prod_{i=0}^{r-1} (1-aq^{-i})$.

\item\label{contantorderfluct} Assume that $k_n\to \infty$ and $\lim_{n\to\infty}\frac{q^{-n/2}}{q-1}k_n=0$. Let $n_1<n_2<\cdots$ be a subsequence of the positive integers such that the fractional part $\{-\log_q(k_{n_i})\}$ converges to $\zeta$. Let $\chi=q^{-\zeta}/(q-1)$. Then for all $j\in \mathbb{Z}$, we have
\[\lim_{i\to\infty}\mathbb{P}\left(\dim \ker C^{(n_i)}_{2k_{n_i}-1} -\lfloor \log_q(k_{n_i})+\zeta \rceil =j\right)=\mathbb{P}(\mathcal{L}_{1,q^{-1},\chi}=j),\]
where $\mathcal{L}_{1,q^{-1},\chi}$ is a random variable defined in~\cite{van2023local} and $\lfloor x\rceil$ denotes the integer closest to $x$.

\end{enumerate}
\end{theorem}

Finally, we also consider the corank of a third sequence of random matrices, which serve as a simplified model of band matrices discussed above. 
Let $\widehat{C}_{2k}$ be a truncated version of $C_{2(k+1)}$ obtained from $C_{2(k+1)}$ by deleting the first $n/2$ rows and the last $n/2$ rows. Note that $\widehat{C}_{2k}$ is a $(k+1)n\times (k+1)n$ matrix. When we want to emphasize the dependence on $n$, we write $\widehat{C}_{2k}^{(n)}$ in place of $\widehat{C}_{2k}$. We think of the ``band width'' of $\widehat{C}_{2k_n}^{(n)}$ as $n/2$, because $n/2$ is the minimum distance between the diagonal and the entries of $\widehat{C}_{2k_n}^{(n)}$ that are set to be $0$.

\begin{figure}[h!]
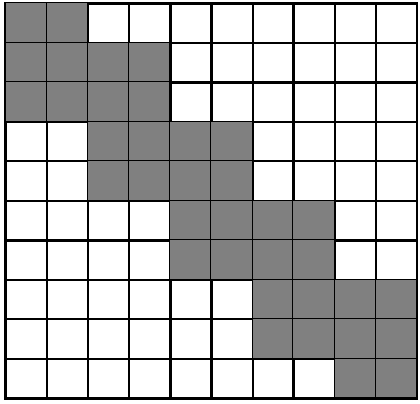

\centering
\def\arraystretch{1.4}
\begin{tabular}{|p{5pt}|p{5pt}|p{5pt}|p{5pt}|p{5pt}|p{5pt}|p{5pt}|p{5pt}|p{5pt}|p{5pt}|} 
\hline
\cellcolor{gray}&\cellcolor{gray}&& &&& &&& \\ 
\hline
\cellcolor{gray}&\cellcolor{gray}&\cellcolor{gray}& \cellcolor{gray}&&& &&& \\
\hline
\cellcolor{gray}&\cellcolor{gray}&\cellcolor{gray}& \cellcolor{gray}&&& &&& \\
\hline
&&\cellcolor{gray}&\cellcolor{gray}&\cellcolor{gray}& \cellcolor{gray}&&& & \\
\hline
&&\cellcolor{gray}&\cellcolor{gray}&\cellcolor{gray}& \cellcolor{gray}&&& & \\
\hline
&&&&\cellcolor{gray}&\cellcolor{gray}&\cellcolor{gray}& \cellcolor{gray}&& \\
\hline
&&&&\cellcolor{gray}&\cellcolor{gray}&\cellcolor{gray}& \cellcolor{gray}&& \\
\hline
&&&&&&\cellcolor{gray}&\cellcolor{gray}&\cellcolor{gray}& \cellcolor{gray} \\
\hline
&&&&&&\cellcolor{gray}&\cellcolor{gray}&\cellcolor{gray}& \cellcolor{gray} \\
\hline
&&&&&&&&\cellcolor{gray}& \cellcolor{gray} \\
\hline

\end{tabular}
\caption{A schematic image of the block matrix $\widehat{C}_8$. All the cells correspond to \break $(n/2)\times (n/2)$ blocks. The entries of the white cells are set to be $0$. The entries of the gray cells are chosen as i.i.d. uniform elements of~$\mathbb{F}_q$.}

\end{figure}

For $u\ge 0$, let $J_u$ be a $\mathbb{Z}_{\ge 0}$-valued random variable such that
\[\mathbb{P}(J_u=k)=q^{-k(k+u)} \prod_{i=1}^k (1-q^{-i})^{-1}\prod_{i=1}^{k+u} (1-q^{-i})^{-1}\prod_{i=1}^\infty (1-q^{-i}).\]

Note that if $M_n$ is a uniform random $(n+u)\times n$ matrix over $\mathbb{F}_q$, then $\dim \ker M_n$ converges to $J_u$ in distribution. In particular, the law of $J_0$ is the same as the one given in~\eqref{universaleqint}.

Consider $Z_t$ with the initial condition $Z_0=0$, and let $D_{0,t}$ be the number of downward jumps made by $Z$ until time $t$. 

Let 
\[L_t=D_{0,t}+|Z_t|_++J_{|Z_t|},\]
where $J_0,J_1,\dots,$ are independent of $Z_t$, and \[|x|_+=\max(x,0),\quad |x|_-=\max(-x,0).\]

As promised, the next theorem describes a Cohen-Lenstra/non-Cohen-Lenstra phase transition on the level of coranks for $ \widehat{C}^{(n)}_{2k_n}$ together with a precise description of the critical behavior.

\begin{theorem}\label{theoremtruncated}
Let $k_n$ be a sequence of positive integers, and let $t_n=\frac{q^{-n/2}}{q-1}k_n$. 
\begin{enumerate}[(a)]
\item \label{theoremtruncatedpart1} Assume that $\lim_{n\to\infty}t_n=0$, then $\dim\ker \widehat{C}^{(n)}_{2k_n}$ converges to $J_0$ in total variation distance.
\item \label{theoremtruncatedpart2} Assume that $\lim_{n\to\infty}t_n=t$, where $0<t<\infty$, then $\dim\ker \widehat{C}^{(n)}_{2k_n}$ converges to $L_t$ in total variation distance.

\item \label{theoremtruncatedpart3}Assume that $\lim_{n\to\infty}t_n=\infty$, then
\[\frac{\dim\ker \widehat{C}^{(n)}_{2k_n}-\mu_n t_n}{\sigma \sqrt{t_n}}\]
converges in distribution to a standard normal variable. Here the constants $\mu_1,\mu_2,\dots$ and $\sigma$ are the same as in Theorem~\ref{thmdimkerC2k}.

\end{enumerate}

\end{theorem}

Observing that $L_0=J_0$, we could merge the cases \eqref{theoremtruncatedpart1} and \eqref{theoremtruncatedpart2} of Theorem~\ref{theoremtruncated} together. We kept the case \eqref{theoremtruncatedpart1} as a separate case to have an explicit description of the Cohen-Lenstra phase, that is, the phase when the asymptotic behavior of the corank of $\widehat{C}^{(n)}_{2k_n}$ cannot be distinguished from that of a square matrix with i.i.d. uniform entries from $\mathbb{F}_q$.

\subsection{Universality in $p$-adic random matrix theory}\label{extandedintro}

There is a widespread universality phenomenon for the distribution of coranks of random matrices over $\mathbb{F}_p$.  The simplest form of this universality was proved by Wood~\cite{wood2019random}: Let $p$ be a fixed prime, and let $\overline{M}_n$ be a random matrix over $\mathbb{F}_p$ where the entries are independent and non-degenerate in a certain sense. Then, as $n$ tends to infinity, $\dim \ker \overline{M}_n$ converges to a universal limiting distribution, which does not depend on the distribution of the entries, only on $p$. More specifically, for any non-negative integer $k$, we have
\begin{equation}\label{universaleq}\lim_{n\to\infty}\mathbb{P}(\dim\ker \overline{M}_n=k)=p^{-k^2} \prod_{i=1}^k (1-p^{-i})^{-2}\prod_{i=1}^\infty (1-p^{-i}).\end{equation}

 See also \cite{kozlov1966rank,kovalenko1975limit,charlap1990asymptotic,kahn2001singularity} for earlier results in this direction. 
 
Equation \eqref{universaleq} is actually a corollary of a more general universality result on the Sylow $p$-subgroups of the cokernels of random matrices. The cokernel of an $n\times n$ matrix $M$ over the integers is defined as the quotient group $\cok(M)=\mathbb{Z}^n/\mathbb{Z}^n M$. If $\det M\neq 0$, then $\cok(M)$ is a finite abelian group of order $|\det M|$. If $\overline{M}$ is the $n\times n$ matrix over $\mathbb{F}_p$ obtained from $M$ by replacing the entries with their mod $p$ remainder, then $\dim\ker \overline{M}$ equals to the rank of the Sylow $p$-subgroup of $\cok(M)$. Thus, $\eqref{universaleq}$ follows from the following more general result of Wood~\cite{wood2019random}: Let $M_n$ be an $n\times n$  random matrix over $\mathbb{Z}$, where the entries are independent and non-degenerate in a certain sense. Then the Sylow $p$-subgroup $\Gamma_n$ of $\cok(M_n)$ has a limiting distribution, that is, given any deterministic finite abelian $p$-group $G$, we have
\begin{equation}\label{IntroductionCohenL}\lim_{n\to\infty} \mathbb{P}(\Gamma_n\cong G)=\frac{1}{|\Aut (G)|}\prod_{i=1}^{\infty}(1-p^{-i}).
\end{equation}
Note that the limit is universal, that is, it does not depend on the distribution of entries. Relying on the above-mentioned identity that $\rang(\Gamma_n)=\dim\ker \overline{M}_n$ one can prove that \eqref{IntroductionCohenL} implies \eqref{universaleq}. The distribution appearing on the right-hand side of \eqref{IntroductionCohenL} is called the Cohen-Lenstra distribution. It first appeared in the conjectures of Cohen and Lenstra on the distribution of the class groups of quadratic number fields~\cite{cohen2006heuristics}. These number theoretic conjectures are still open. As the class group can be obtained as a cokernel~\cite{wood2022number,venkatesh2010statistics}, these conjectures served as the initial motivation for investigating the cokernels of random matrices. When Wood developed the moment method for random abelian groups~\cite{wood2017distribution}, it really sparked the growth of the area of cokernels of random matrices. In recent years, using the moment method, it was proved for several families of random matrices that their cokernels asymptotically follow the Cohen-Lenstra distribution or some close variant of it~\cite{nguyen2022random,recent2,recent3,recent4,recent5,recent6,recent7,recent8,meszaros2020distribution,meszaros2023cohen,meszaros2024Zpband,gorokhovsky2024time,kang2024random}, see also the survey of Wood~\cite{wood2022probability}. Although the initial motivation behind the investigation of the
cokernels of random matrices comes from number theory, in the recent years new lines of research
arose motivated by probability theory~\cite{assiotis2022infinite,van2023p,van2023local,van2021limits,lvov2024random,meszaros2024Zpband,shen2024gaussian,shen2025gaussian,shen2024non}. 

It is often more convenient to consider the cokernel of matrices over the $p$-adic integers $\mathbb{Z}_p$ instead of matrices over~$\mathbb{Z}$. For matrices over $\mathbb{Z}_p$, the cokernel is a $p$-group, so we do not need to take the Sylow $p$-group of the cokernel. Another advantage of working over $\mathbb{Z}_p$ is that $\mathbb{Z}_p$ has a finite Haar measure as opposed to $\mathbb{Z}$. The results of Wood also apply in the $p$-adic setting.

It is natural to ask how universal the Cohen-Lenstra behavior of the cokernel is. It is known that if we impose some additional algebraic structure on our random matrix model, then this has an effect on the cokernel. For example, for symmetric matrices, the cokernel can be equipped with a perfect, symmetric, bilinear pairing. Thus, to obtain the limiting distribution for the cokernels of symmetric matrices, one needs to modify the Cohen-Lenstra distribution by taking into account the number of such pairings~\cite{clancy2015note,clancy2015cohen,wood2017distribution}. A similar phenomenon occurs for skew symmetric matrices~\cite{delaunay2001heuristics,bhargava2015modeling,nguyen2024local}, and for $p$-adic Hermitian matrices~\cite{lee2023universality}. We again see that the limiting behavior of the cokernel is universal: it depends on the symmetry class of the random matrix model, but not on the specific distribution of the entries. 

Another way one can try to break away from the Cohen-Lenstra universality class is to consider matrices with some underlying geometric structure. Due to the non-trivial geometric structure of these matrices, we can observe various new types of behavior of the corank which are significantly different from the Cohen-Lenstra behavior discussed above. For example, the homology group of a random simplicial complex can be basically obtained as the cokernel of its boundary matrix. In the case of the first homology group, the rows of the boundary matrix are indexed by the edges of a graph. The presence of this rigid graph structure makes it more difficult to understand cokernels of these boundary matrices compared to the unstructured matrices that we usually investigate~\cite{kahle2020cohen,meszaros20242,meszaros2024bounds,meszaros2025using}. Another way to inject some geometric structure into our random matrix models is to fix certain entries to be $0$ based on some geometric pattern~\cite{meszaros2024Zpband,kang2024random}.  The three random matrix models considered in this paper fit into this line of inquiry.   

One of the first results on the cokernels of structured random matrices is by the author, who considered Haar-uniform band matrices over $\mathbb{Z}_p$ and proved a Cohen-Lenstra/non-Cohen-Lenstra phase transition for them~\cite{meszaros2024Zpband}. This result is motivated by the analogous questions about Gaussian band matrices in classical random matrix theory, as we explain next: A Gaussian $n\times n$ band matrix with band width $w$ is a random symmetric matrix $X_{n,w}$, where $X_{n,w}(i,j)=0$ for all $i,j$ such that $|i-j|>w$, and all the other entries above and on the diagonal are chosen to be i.i.d. standard Gaussians. It is conjectured that band matrices exhibit a phase transition \cite{bourgade2018random}:
\begin{itemize}
 \item If $w\ll \sqrt{n}$, then the eigenvectors of $X_{n,w}$ are \emph{localized}, and the eigenvalues have \emph{Poisson local statistics}. 
 \item If $w\gg \sqrt{n}$, then the eigenvectors of $X_{n,w}$ are \emph{delocalized}, and the eigenvalues have \emph{GOE local statistics}. 
\end{itemize}

In physical terms, this corresponds to a metal/insulator phase transition. The localization of eigenvectors was proved for $w\ll n^{1/8}$ in~\cite{schenker2009eigenvector}. Later, the exponent $1/8$ was improved to ${1/7}$ and then to ${1/4}$ \cite{peled2019wegner,cipolloni2024dynamical,chen2022random}. Finally, the conjectured exponent of $1/2$ was achieved very recently by Drogin~\cite{drogin2025localization}. For fixed $w$, Poisson eigenvalue statistics were proved in~\cite{brodie2022density}, see also~\cite{hislop2022local} for further progress towards Poisson statistics. Improving earlier results \cite{band1,band2,band3,band4,band5,bb1,bb2,bb3}, the delocalization of eigenvectors and the GOE statistics were proved for $w\gg n^{1/2}$ by Yau, Yin \cite{yau2025delocalization} and by Erd\H{o}s, Riabov~\cite{erdHos2025zigzag}.

Note that the cokernel of a matrix can also be read off from its Smith normal form~\cite{stanley2016smith}. Thus, similarly to the spectrum, the cokernel can also be obtained by some sort of diagonalization of the matrix. This suggests that the cokernel can be viewed as an analogue of the spectrum, and in this analogy, the Cohen-Lenstra distributed cokernel should correspond to GOE eigenvalue statistics. This analogy combined with the above-mentioned conjectures on Gaussian band matrices predicts a Cohen-Lenstra/non-Cohen-Lenstra phase transition for the cokernels of band matrices in the $p$-adic setting. Such a result was proved by the author~\cite{meszaros2024Zpband}: Let $B_n$ be a Haar-uniform $n\times n$ band matrix over $\mathbb{Z}_p$ with band width $w_n$, then $\cok(B_n)$ has Cohen-Lenstra limiting distribution if and only if $\lim_{n\to\infty} np^{-w_n}=0$. The author also proved results which can be considered as the $p$-adic analogue of the localization/delocalization phase transition in the Gaussian case. Section~\ref{intlocsubsection} contains the analogous results for the random matrix model $\widehat{C}_{2k_n}$.

The author conjectured that at criticality a new one-parameter family of limiting distributions arises:
\begin{conjecture}\label{conjecturecritical}
There is a one-parameter family of distributions $(\nu_{t})_{t\in\mathbb{R}_+}$ on the set of finite abelian $p$-groups with the following property. Let $n_1<n_2<\dots$ be a sequence of positive integers. Let ${B}_i$ be a Haar-uniform $n_i\times n_i$ band matrix over $\mathbb{Z}_p$ with band width $w_i$. Let us assume that $\lim_{i\to\infty} p^{-w_i} n_i=t$. Then for a finite abelian $p$-group $G$, we have
\[\lim_{i\to\infty} \mathbb{P}(\cok({B}_i)\cong G)=\nu_t(G).\]
\end{conjecture}

When we are only interested in the corank, Conjecture~\ref{conjecturecritical} reduces to Conjecture~\ref{conjecturecriticalInt}.

\subsection{A localization/delocalization phase transition for the kernel of $\widehat{C}_{2k_n}$}\label{intlocsubsection}

Given a vector $v\in \mathbb{F}_q^{(k+1)n}$, we can uniquely write $v$ as the concatenation of some vectors ${v_1,\dots,v_{k+1}\in \mathbb{F}_q^n}$. For $v\in \mathbb{F}_q^{(k+1)n}$, let us define,
\[\supp (v)=\{i\in \{1,2,\dots,k+1\}\,:\, v_i\neq 0\}.\]

We say that $\widehat{C}_{2k_n}$ has a delocalized kernel if $\supp(v)=\{1,2,\dots,k_n+1\}$ for all $v\in {\ker \widehat{C}_{2k_n}\setminus\{0\}}$.

Given a positive integer $L $, we set $\bar{k}_n=\left\lfloor\frac{k_n+1}L \right\rfloor$, and for $i=1,2,\dots,L$, we define
\[V_i=\left\{v\in \mathbb{F}_q^{(k_n+1)n}\,:\,\emptyset\neq \supp v\subset \left[(i-1)\bar{k}_n+1\,,\,i \bar{k}_n\right]\right\}.\]

We say that $\widehat{C}_{2k_n}$ has an $L $-localized kernel if the following event occurs:
\[\text{For all }i=1,2,\dots,L, \text{ there is a }v_i\in V_i\text{ such that }\widehat{C}_{2k_n} v_i=0.\]

\begin{theorem}\label{thmlocdeloc} Let $k_n$ be a sequence of positive integers, and let $t_n=\frac{q^{-n/2}}{q-1}k_n$. 
\begin{enumerate}[(1)]
 \item\label{partdeloc} Assume that $\lim_{n\to\infty} t_n=0$, then
 \[\lim_{n\to\infty} \mathbb{P}(\,\widehat{C}_{2k_n}^{(n)}\text{ has a delocalized kernel}\,)=1.\]
 \item\label{partloc} Assume that $\limsup_{n\to\infty} t_n\ge \varepsilon>0$, then for each $L\ge \varepsilon$, we have 
 \[\limsup_{n\to\infty} \mathbb{P}(\,\widehat{C}_{2k_n}^{(n)}\text{ has an $L$-localized kernel}\,)\ge \left(C\left(\frac{\varepsilon}L\right)^2\right)^{L}\]
 for some constant $C$ depending only on $q$.
\end{enumerate}
\end{theorem}

Note that if $\widehat{C}_{2k_n}$ has an $L$-localized kernel, then $\dim \ker \widehat{C}_{2k_n}\ge L$. Thus, it follows that if $\lim_{n\to\infty} \frac{q^{-m}}{q-1}k_n=t>0$, then for all $L>t$
\[\lim_{n\to\infty} \mathbb{P}(\dim \ker \widehat{C}_{2k_n}\ge L)\ge \left(C\left(\frac{t}L\right)^2\right)^{L}=\exp(-O_t(L\log L)).\]
If we compare this with $\mathbb{P}(J_0\ge L)=\exp(-\Theta(L^2))$, we see that the limiting distribution of $\dim \ker \widehat{C}_{2k_n}$ has a much heavier tail than the distribution of the rank of a Cohen-Lenstra distributed group.

Thus, Theorem~\ref{thmlocdeloc} provides a setting where the heavier than Cohen-Lenstra tail of the corank can be explained by the presence of vectors with small support in the kernel. See \cite{meszaros20242,meszaros2024Zpband,kang2024random} for other occurrences of this phenomenon. It would be nice to build a general theory that can explain why the heavier than Cohen-Lenstra tail of the corank and the localization of the kernel tend to appear together. At this point we do not even have a clear general definition of localization.

\subsection{The main ideas of the proofs -- Coupling of Markov chains}\label{SecMainIdeas}

We will understand the behavior of $\rang(C_i)$ from a dynamical viewpoint, that is, we will understand the evolution of $\rang(C_i)$ as $i$ grows.

Let $X_0=0$ and for $i>0$, let
\begin{equation}\label{rankincrementdef}X_i=\rang(C_i)-\rang(C_{i-1}),\end{equation}
where $\rang(C_0)$ is defined to be $0$.

Our results rely on the fact that $X_i$ is a Markov chain with explicitly given transition probabilities, as the following lemma shows.

\begin{lemma}\label{lemmaMarkov}
The sequence $(X_i)$ is a time homogeneous irreducible aperiodic reversible Markov chain on the state space $\{0,1,\dots,n\}$ with transition matrix
\begin{equation}\label{lemmaMarkoveq}P_n(d,r)=\mathbbm{1}(d+r\le n)q^{-(n-d-r)(n-r)} \frac{(q^{-(n-d)};q)_r (q^{-n};q)_r}{(q^{-r};q)_r},
\end{equation}
where we used the $q$-Pochhammer symbol $(a;q)_r=\prod_{i=0}^{r-1} (1-aq^i)$.

Moreover, the stationary measure $\pi_n$ of this Markov chain is given by

\[\pi_n(h)=\frac{1}{\mathcal{Z}_n}q^{h(n-h)}\frac{(q^{-n};q)_h^2}{(q^{-h};q)_h}, \quad\text{ where }\quad\mathcal{Z}_n=\sum_{h=0}^n q^{h(n-h)}\frac{(q^{-n};q)_h^2}{(q^{-h};q)_h}.\]

\end{lemma}

Roughly speaking, most of our results rely on the fact that the appropriately rescaled and recentered versions of the Markov chains in Lemma~\ref{lemmaMarkov} can be coupled with the continuous time random walk $Z_t$ such that they are close to each other with high probability. We provide the details of this coupling result next.

As before, we assume that $n=2m$ is even. See Section~\ref{oddn} for the necessary modifications in the case of an odd $n$.

Let $X_i$ be a Markov chain with transition probabilities given by $P_n$. The distribution of $X_0$ is to be specified later. If $X_0=a$ for some deterministic constant $a$, then sometimes we will write $X_{a,i}$ in place of $X_i$. Let us consider the following rescaled and recentered version $(Y_t)_{t\in \mathbb{R}_{\ge 0}}$ of $X_i$:
\[Y_t=X(2\lfloor q^{n/2}(q-1) t\rfloor)-n/2.\]

Recall that we assume that $n$ is even, so $Y_t$ is integer valued.

With some effort, one can see that typically $X_i+X_{i+1}$ is very close to~$n$. Thus, if $X_i$ is far away from $\frac{n}2$, then there is a significant gap between $X_i$ and $X_{i+1}$. However, even in this case the values of $X_i$ and $X_{i+2}$ are close to each other. This is the reason why, in the definition of $Y_t$, we only consider $X_i$ at even indices.

Let $R_0,R_1,\dots$ be the jump times of $Y_t$, that is, $R_0=0$ and
\begin{equation}\label{Ridef}R_i=\inf\{t>R_{i-1}\,:\,Y_t\neq Y_{R_{i-1}}\}.\end{equation}
Finally, let $\widehat{Y}_i=Y_{R_i}$.

When we want to emphasize the dependence on $n$, we write $X_i^{(n)},Y_i^{(n)},\widehat{Y}_i^{(n)}, R_i^{(n)}$ and so on.

Let $T_0,T_1,\dots$ be the jump times of $Z_t$, that is, $T_0=0$ and
\[T_i=\inf\{t>T_{i-1}\,:\,Z_t\neq Z_{T_{i-1}}\}.\]

Note that $Z_t$ has no explosions, that is, $T_i\to\infty$ almost surely, see Lemma~\ref{noexplosion}.

Let $\widehat{Z}_i$ be the embedded discrete time Markov chain of $Z_t$, that is, $\widehat{Z}_i=Z_{T_i}$.

The next key theorem states that the processes $Y^{(n)}$ converge to $Z_t$.

\begin{theorem}\label{theoremcoupling}
Let $\alpha=0.08$.  Let $a$ be any deterministic constant such that $|a|\le 2\alpha n$. Assume that $Y^{(n)}_0=Z_0=a$. Then there is a coupling of the processes $Z_t$ and $Y_t=Y_t^{(n)}$ such that with probability at least $1-O\left(q^{-2\alpha n}\right)$ the following holds:
\begin{enumerate}[(a)]
\item For all $0\le i\le q^{2\alpha n}$, we have $\widehat{Y}_i=\widehat{Z}_i$;
\item For all $0\le i\le q^{2\alpha n}$, we have $|R_i-T_i|\le q^{-\alpha n}$;
\item $T_{\lfloor q^{2\alpha n} \rfloor}\ge q^{\alpha n}$.
\end{enumerate}
\end{theorem}

The proof relies on the observation that for any $a$ with small absolute value, we have
\[\mathbb{P}(X_{2i+2}^{(n)}=m+b\,|\,X_{2i}^{(n)}=m+a)\approx \begin{cases} \frac{q^{-a}}{q^{-a}+q^{a}}\left(1-\exp\left(-\frac{q^{-m}}{q-1}(q^{-a}+q^{a}) \right)\right)
&\text{ if }$b=a+1$,\\
\frac{q^{a}}{q^{-a}+q^{a}}\left(1-\exp\left(-\frac{q^{-m}}{q-1}(q^{-a}+q^{a}) \right)\right)
&\text{ if }$b=a-1$,\\
\exp\left(-\frac{q^{-m}}{q-1}(q^{-a}+q^{a}) \right)
&\text{ if }$b=a$,\\
\text{negligible}&\text{ otherwise.}
\end{cases},\]
where we intentionally do not specify the meaning of the symbol $\approx$ at this point. From the formula above, it follow that the Markov-chains $\widehat{Y}_i$ and $\widehat{Z}_i$ asymptotically have the same transition probabilities. Moreover, conditioned the event that $Y_t=a$ the time until the next jump is given by a rescaled geometric random variable, which in the limit becomes an exponential random variable with rate $q^{-a}+q^{a}$. Although essentially this is the main idea of the proof of Theorem~\ref{theoremcoupling}, further technical work is needed to have sufficient control over all the error terms and get to the full strength of Theorem~\ref{theoremcoupling}.

Theorem~\ref{theoremcoupling} provides a much stronger convergence result than just the convergence of finite dimensional marginals, which is given in the following corollary.
\begin{corollary}\label{marginalconv}
 Let $0<t_1<t_2<\cdots<t_k$ be a deterministic sequence of times. Assume that ${Z}_0$ is constant, and assume that for all large enough $n$, we have $Y_0^{(n)}=Z_0$. Then
$\left(Y_{t_1}^{(n)},Y_{t_2}^{(n)},\dots,Y_{t_k}^{(n)}\right)$ converges to $\left(Z_{t_1},Z_{t_2},\dots,Z_{t_k}\right)$ in total variation distance.
\end{corollary}

For $a\in \mathbb{Z}$ and $t>0$, assuming the initial condition $Z_0=a$, let $D_{a,t}$ be the number of downward jumps made by $Z$ until time $t$ (including a possible jump at time $t$ provided that it is a downward jump).

When $X_i^{(n)}$ is defined as in \eqref{rankincrementdef}, then the initial condition of the process $Y^{(n)}$ must be set to $Y_0^{(n)}=-\frac{n}2$. Thus, the initial condition varies with $n$. Since $\lim_{n\to\infty}Y_0^{(n)}=-\infty$, naturally, we need to understand the limiting behavior of $Z_{a,t}$ as $a\to-\infty$.

\begin{lemma}\label{Dinftyexists}
For any fixed $t>0$, as $a$ tends to $-\infty$, $(D_{a,t},Z_{a,t})$ converge in total variation distance to some random variable $(D_{-\infty,t},Z_{-\infty,t})$. 
\end{lemma}

Finally, let us explain why we care about the downward jumps of $Y_t$. Assume that we are in the critical case of Theorem~\ref{thmdimkerC2k}, that is, $t_n=\frac{q^{-n/2}}{q-1}k_n$ converges to $t\in (0,\infty)$. It turns out with high probability, $X^{(n)}_{2i}\ge X^{(n)}_{2i-2}-1$ and $n-X^{(n)}_{2i-1}-X^{(n)}_{2i}\in \{0,1\}$ for all $1\le i\le k_n$. Moreover, asymptotically the event that $n-X^{(n)}_{2i-1}-X^{(n)}_{2i}=1$ coincides with the event that $X^{(n)}_{2i}=X^{(n)}_{2i-2}-1$. Therefore,
\[\dim\ker C^{(n)}_{2k_n}=k_n n-\rang\left(C^{(n)}_{2k_n}\right)=\sum_{i=1}^{k_n}(n-X^{(n)}_{2i-1}-X^{(n)}_{2i})\]
can be approximated by
\[\sum_{i=1}^{k_n} \mathbbm{1}(X^{(n)}_{2i}=X^{(n)}_{2i-2}-1),\]
that is, with the number of downward jumps of $Y_t$ up to time $t_n$. Finally, the number of downward jumps of $Y_t$ up to time $t_n$ can be approximated by $D_{-\infty,t}$. 

\subsection{The structure of the paper}

In Section~\ref{SecMarkov}, we prove Lemma~\ref{lemmaMarkov}, that is, we show that the increments of the ranks of block bidiagonal matrices form a Markov chain. Then, in Section~\ref{SecCoupling}, we show that the appropriately rescaled and recentered versions of these Markov chains can be coupled with the continuous time random walk $Z_t$ such that they stay close to each other sufficiently long. The later sections build heavily on this coupling result to provide a description of the asymptotic evolution of the corank of block bidiagonal matrices. In order to do this, we also need a better understanding of the limiting process $Z_t$. In particular, we need estimates on the number of jumps made by $Z_t$ in any given time interval. Such an estimate is provided in Section~\ref{SecZt}, where we also prove Lemma~\ref{Dinftyexists} on the existence of the limit of $(D_{a,t},Z_{a,t})$ as $a\to-\infty$. The behavior of the corank of $C^{(n)}_{2k_n}$ and $\hat{C}^{(n)}_{2k_n}$ at criticality is understood in Section~\ref{Seccrit}. Adapting the usual proof of the Markov Chain Central Limit Theorem, the Gaussian fluctuation of the corank of $C^{(n)}_{2k_n}$ and $\hat{C}^{(n)}_{2k_n}$ in the supercritical regime is proved in Section~\ref{SecGauss}. The localization/delocalization phase transition for the vectors in the kernel is treated in Section~\ref{SectionLocalization}. In Section~\ref{secproductvsbidiagonal}, we explain how matrix products are related to block bidiagonal matrices and, based on this correspondence, we prove part~\eqref{thmdimkerC2k-1part1} of Theorem~\ref{thmdimkerC2k-1}. The other statements of Theorem~\ref{thmdimkerC2k-1} are proved in Section~\ref{secrestof}. Finally, in Section~\ref{oddn}, we discuss what happens if we assume that $n$ is odd instead of even.

\section{The Markovian property of the rank increments -- The proof of Lemma~\ref{lemmaMarkov}}\label{SecMarkov}

Let $U_1$ be the subspace of $\mathbb{F}_q^{n(k+1)}$ generated by the first $kn$ columns of $C_{2k+1}$, and let $U_2$ be the subspace of $\mathbb{F}_q^{n(k+1)}$ generated by all the columns of $C_{2k+1}$. Let $R:\mathbb{F}_q^{n(k+1)}\to \mathbb{F}_q^{nk}$ be the projection onto the first $nk$ components. Let $i=1$ or $2$. Applying the rank-nullity theorem to the restriction of $R$ to $U_i$, we obtain that
\[\dim U_i=\dim (U_i\cap \ker R)+\dim (\range (R\restriction U_i)).\]

Note that $\range (R\restriction U_i)$ is just the range of the matrix $C_{2k-1}$ for both choices of $i$. Thus, $\dim (\range R\restriction U_i)=\rang C_{2k-1}$. Also, $\dim U_1=\rang C_{2k}$ and $\dim U_2=\rang C_{2k+1}$. Thus, we get the following equations
\begin{align}
\rang C_{2k}&=\dim (U_1\cap \ker R)+\rang C_{2k-1},\label{delta1}\\
\rang C_{2k+1}&=\dim (U_2\cap \ker R)+\rang C_{2k-1}\nonumber.
\end{align}
Subtracting these two equations, it follows that
\[X_{2k+1}=\rang C_{2k+1}-\rang C_{2k}=\dim (U_2\cap \ker R)-\dim (U_1\cap \ker R).\]

Let $V$ be the subspace of $\mathbb{F}_q^{n(k+1)}$ generated by the last $n$ columns of $C_{2k+1}$. Observe that $V\subset \ker R$. Thus, \[U_2\cap \ker R=(U_1+V)\cap \ker R=(U_1\cap \ker R)+V.\] 

Let $Q:\mathbb{F}_q^{n(k+1)}\to \mathbb{F}_q^{n}$ be the projection onto the last $n$ components. Clearly, $Q$ is an isomorphism when restricted to $\ker R$, therefore $\dim (U_i\cap \ker R)=\dim (\range (Q\restriction (U_i\cap \ker R)))$. Also, $\range (Q\restriction V)=\range A_{k+1,k+1}$. Let $W=\range (Q\restriction(U_1\cap \ker R))$. It follows that
\[X_{2k+1}=\dim (W+\range A_{k+1,k+1})-\dim (W)=\dim((\range A_{k+1,k+1}+W)/W).\]

Also, by \eqref{delta1}, we see that $\dim W=\rang C_{2k}-\rang C_{2k-1}=X_{2k}$.

Let $\tilde{A}_{k+1,k+1}:\mathbb{F}_q^n\to \mathbb{F}_q^n/W$ defined by $\tilde{A}_{k+1,k+1}(v)=A_{k+1,k+1}v+W$, then 
\[X_{2k+1}=\dim\range \tilde{A}_{k+1,k+1} .\]

Note that $A_{k+1,k+1}$ is a uniform random $n\times n$ matrix over $\mathbb{F}_q$ and it is independent of $W$. Thus, $\tilde{A}_{k+1,k+1}$ will be a uniform random linear map from $\mathbb{F}_q^n$ to $\mathbb{F}_q^n/W$. By choosing bases, $\tilde{A}_{k+1,k+1}$ can be identified with a uniform random $(n-\dim W)\times n$ matrix. Thus, recalling that $\dim W=X_{2k}$, we get that
\begin{align*}\mathbb{P}(X_{2k+1}=r\,|\, X_{2k}=d)&=\mathbb{P}(\text{A uniform random $(n-d)\times n$ matrix over $\mathbb{F}_q$ has rank }r)\\
&=\mathbbm{1}(d+r\le n)q^{-(n-d-r)(n-r)} \frac{(q^{-(n-d)};q)_r (q^{-n};q)_r}{(q^{-r};q)_r},
\end{align*}
where the last equality follows from the following well-known lemma. For the reader's convenience, we provide a proof.
\begin{lemma}
For $d\ge 0$, let $B$ be a uniform random $(n-d)\times n$ matrix over $\mathbb{F}_q$, then for all $r\le n-d$, we have
\[\mathbb{P}(B\text{ has rank }r)=q^{-(n-d-r)(n-r)} \frac{(q^{-(n-d)};q)_r (q^{-n};q)_r}{(q^{-r};q)_r}.\]
\end{lemma}
\begin{proof}
Let $V$ be an $r$-dimensional subspace of $\mathbb{F}_q^{n-d}$. We claim that 
\begin{equation}\label{rangeequalformula}
\mathbb{P}(\range B=V)=q^{-(n-d-r)n}\prod_{i=0}^{r-1}\frac{q^n-q^i}{q^n}.
\end{equation}
By symmetry, it is enough to prove this when \[V=\{v\in \mathbb{F}_q^{n-d}:v_i=0\text{ for all }i>r\}.\]
In this case, the event $\range B=V$ is the same as the event that the first $r$ rows of $B$ are linearly independent and the last $n-d-r$ rows of $B$ are all equal to $0$. Thus, \eqref{rangeequalformula} follows easily.

The number of $r$-dimensional subspaces $V$ of $\mathbb{F}_q^{n-d}$ is given by the Gaussian binomial coefficient ${{n-d}\choose{r}}_q$. Therefore, the statement follows easily.
\end{proof}

\begin{remark}\label{remarkI}
 If $A_{k+1,k+1}$ is the identity matrix, then $\range \tilde{A}_{k+1,k+1}=\mathbb{F}_q^n/W$. Thus, 
 \[X_{2k+1}=\dim \range \tilde{A}_{k+1,k+1}=\dim (\mathbb{F}_q^n/W)=n-\dim W=n-X_{2k}.\]
 Similarly, one can prove that if $A_{k+2,k+1}=I$, then $X_{2k+2}=n-X_{2k+1}$.
\end{remark}

Almost the same argument can be used to prove that 
\[\mathbb{P}(X_{2k+2}=r\,|\, X_{2k+1}=d)=\mathbbm{1}(d+r\le n)q^{-(n-d-r)(n-r)} \frac{(q^{-(n-d)};q)_r (q^{-n};q)_r}{(q^{-r};q)_r}.\]

We just need to redefine $U_1$ and $U_2$ as follows. Let $U_1$ be the subspace of $\mathbb{F}_q^{n(k+1)}$ generated by the first $(k+1)n$ rows of $C_{2k+2}$, and let $U_2$ be the subspace of $\mathbb{F}_q^{n(k+1)}$ generated by all the rows of $C_{2k+2}$.

\begin{remark}\label{remarkOthern}
The same argument as above gives that if $C'_{2k+2}$ is obtained from $C_{2k+2}$ by replacing the $n\times n$ block $A_{k+2,k+1}$ by a uniform random $n'\times n$ block, then 
\begin{align*}\mathbb{P}&\left(\rang(C'_{2k+2})-\rang(C_{2k+1})=r\,|\, X_{2k+1}=d\right)\\&\qquad=\mathbb{P}(\text{A uniform random $n'\times(n-d)$ matrix over $\mathbb{F}_q$ has rank }r)\\
&\qquad=\mathbbm{1}(r\le \min(n',n-d))q^{-(n-d-r)(n'-r)} \frac{(q^{-(n-d)};q)_r (q^{-n'};q)_r}{(q^{-r};q)_r}.
\end{align*}

\end{remark}

\bigskip

Thus, we proved that $(X_i)$ is indeed a Markov chain with the transition matrix $P_n$ given in~\eqref{lemmaMarkoveq}. This Markov chain is clearly irreducible and aperiodic. 

Next, we prove that it is also reversible by applying Kolmogorov's criterion. 

Let $v_0,v_1,\dots,v_\ell=v_0$ be a cycle. We need to prove that 
\begin{equation}\prod_{i=1}^{\ell} P_n(v_{i-1},v_i)=\prod_{i=1}^{\ell} P_n(v_{i},v_{i-1}).\label{Kolmogorov}\end{equation}

It is more convenient to rewrite $P_n(d,r)$ as

\[P_n(d,r)=\mathbbm{1}(d+r\le n)q^{-(n-d)(n-r)+r(n-r)} \frac{(q^{-n};q)_{r+d} (q^{-n};q)_r}{(q^{-n};q)_d(q^{-r};q)_r}.\]
If $v_{i-1}+v_i>n$ for some $i$, then both sides of \eqref{Kolmogorov} are $0$. Otherwise, both sides of \eqref{Kolmogorov} are equal to 
\[\prod_{i=1}^\ell q^{-(n-v_i)(n-v_{i-1})+v_i(n-v_i)} \frac{(q^{-n};q)_{v_i+v_{i-1}} }{(q^{-v_i};q)_{v_i}}.\]
Thus, we proved reversibility.

As a corollary of reversibility, the stationary distribution $\pi_n$ satisfies
\[\pi_n(h)=\frac{P_n(0,h)}{P_n(h,0)}\pi_n(0)=q^{h(n-h)}\frac{(q^{-n};q)_h^2}{(q^{-h};q)_h}\pi_n(0).\]

Thus, the last statement of Lemma~\ref{lemmaMarkov} follows.

\section{Constructing a coupling of $Z_t$ and $Y_t$ -- The proof of Theorem~\ref{theoremcoupling}}\label{SecCoupling}

Throughout the paper, we use the notation $m=n/2$.

Let $Q_0=0$, and for $j\ge 0$, let 
\begin{equation}\label{Qidef}Q_{j+1}=\min\{i>Q_j\,:\, X_{2i}\neq X_{2(i-1)}\}.\end{equation}

Note that $R_j=\frac{q^{-m}}{q-1}Q_j$.
We also define 
\begin{equation}\label{hatXdef}\widehat{X}_j=X_{2Q_j}-m.\end{equation}

Note that $\widehat{X}_j=\widehat{Y}_j$ for all $j$.

The next theorem is clearly stronger than Theorem~\ref{theoremcoupling}.

\begin{theorem}\label{theoremcoupling0}
Assume that $\widehat{X}_0=Z_0=a$, where $a$ is a deterministic constant satisfying $|a|\le 2\alpha n$. Then there is a coupling of the processes $X_i$ and $Z_t$ such that with probability at least $1-O\left(q^{-2\alpha n}\right)$ the following holds:
\begin{enumerate}[(a)]
\item \label{couplingevent1} For all $0\le i\le q^{2\alpha n}$, we have $\widehat{X}_i=\widehat{Z}_i$;
\item \label{couplingevent2}For all $0\le i\le q^{2\alpha n}$, we have $\left|\frac{q^{-m}}{q-1}Q_i-T_i\right|\le q^{-\alpha n}$;
\item \label{couplingevent3}$T_{\lfloor q^{2\alpha n} \rfloor}\ge q^{\alpha n}$,
\item \label{couplingeventc}For all $1\le i\le Q_{\lfloor q^{2n\alpha}\rfloor}$, $|n-X_{2i-2}-X_{2i-1}|+|n-X_{2i-1}-X_{2i}|\le 1$;
\end{enumerate}
\end{theorem}

Note that the events \eqref{couplingevent1}, \eqref{couplingevent2} and \eqref{couplingevent3} in Theorem~\ref{theoremcoupling0} directly translate to the events in Theorem~\ref{theoremcoupling}. However, Theorem~\ref{theoremcoupling0} contains the additional event~\eqref{couplingeventc} which cannot be expressed only in terms of the process $Y_t$. The significance of this event will be clear from \eqref{Dpasdown}.

The rest of the section is devoted to the proof of Theorem~\ref{theoremcoupling0}.

Given two expressions $A$ and $B$, we write $A=O(B)$ if there is a constant $C$ such that $|A|\le CB$, and we write $A=O_{0\le }(B)$ if there is a constant $C$ such that $0\le A\le CB$. The constant $C$ may depend on $q$, but it cannot depend on any other variable involved in the expressions $A$ and $B$. 

The following lemma is a simple corollary of Taylor's theorem and the mean value theorem.

\begin{lemma}\label{taylor}\hfill
\begin{enumerate}[(a)]
    \item Assuming that $x=O(1)$, we have $\exp(x)=1+x+O_{0\le}(x^2)$.
    \item Assuming that $-\frac{1}2\le x\le O(1)$, we have \[\log(1+x)=x-O_{0\le}(x^2)\text{, or equivalently }1+x=\exp\left(x-O_{0\le}(x^2)\right).\]
    \item \label{partmean}For any $x$, we have $1-\exp(x)=-x\exp(O(|x|))$.
\end{enumerate}
\end{lemma}

\begin{lemma}
We have
\begin{align}
P_n(d,n-d)P_n(n-d,d)&=\exp\left(-\frac{q^{-d}+q^{-(n-d)}}{q-1}+O\left(q^{-2\min(d,n-d)}\right)\right),\label{approx1}\\[10pt]
\frac{P_n(d,n-d-1)P_n(n-d-1,d+1)}{1-P_n(d,n-d)P_n(n-d,d)}&=\frac{q^{-(d-m)}}{q^{d-m}+q^{-(d-m)}}\exp\left(O\left(q^{-\min(d,n-d)}\right)\right),\label{approx2}\\[20pt]
\frac{P_n(d,n-d)P_n(n-d,d-1)}{1-P_n(d,n-d)P_n(n-d,d)}&=\frac{q^{d-m}}{q^{d-m}+q^{-(d-m)}}\exp\left(O\left(q^{-\min(d,n-d)}\right)\right).\label{approx3}
\end{align}
In \eqref{approx2}, we assume that $d<n$. In \eqref{approx3}, we assume that $0<d$.
\end{lemma}
\begin{proof}

Using Lemma~\ref{taylor}, we see that
\begin{align}
P_n(d,n-d)&=\prod_{i=0}^{n-d-1}\left(1-q^{-(n-i)}\right)\label{appr2}\\
&=\exp\left(-\sum_{i=0}^{n-d-1} \left(q^{-(n-i)}+O_{0\le}\left(q^{-2(n-i)}\right)\right)\right)\nonumber\\&=\exp\left(-q^{-n}\frac{q^{n-d}-1}{q-1}-O_{0\le}\left(q^{-2d}\right)\right)\nonumber\\&=\exp\left(-\frac{q^{-d}}{q-1}-O_{0\le}\left(q^{-2d}\right)+O\left(q^{-n}\right)\right)\nonumber\\&=\exp\left(O\left(q^{-d}\right)\right)\nonumber.
\end{align}

Thus,
\begin{equation}\label{appr25}
P_n(d,n-d)P_n(n-d,d)=\exp\left(-\frac{q^{-d}+q^{-(n-d)}}{q-1}-O_{0\le}\left(q^{-2\min(d,n-d)}\right)+O(q^{-n})\right).
\end{equation}

Equation~\eqref{approx1} follows from \eqref{appr25} and the fact that $n\ge 2\min(d,n-d)$.

Combining~\eqref{appr25} and part \eqref{partmean} of Lemma~\ref{taylor}, we see that
\begin{align}
1-&P_n(d,n-d)P_n(n-d,d)\label{appr29}\\&=1-\exp\left(-\frac{q^{-d}+q^{-(n-d)}}{q-1}-O_{0\le}\left(q^{-2\min(d,n-d)}\right)+O(q^{-n})\right)\nonumber\\&=\left(\frac{q^{-d}+q^{-(n-d)}}{q-1}+O_{0\le}\left(q^{-2\min(d,n-d)}\right)+O(q^{-n})\right)\exp\left(O\left(q^{-\min(d,n-d)}\right)\right).\nonumber
\end{align}

It is easy to see that if we want to prove \eqref{approx2} and \eqref{approx3}, we may assume that $n$ is large enough. If $n$ is large enough, then an expression of the form $O_{0\le}\left(q^{-\min(d,n-d)}\right)+O(q^{-n+\min(d,n-d)})$ is at least $-\frac{1}2$. Thus, Lemma~\ref{taylor} can be applied to give that

\begin{align*}
  \frac{q^{-d}+q^{-(n-d)}}{q-1}&+O_{0\le}\left(q^{-2\min(d,n-d)}\right)+O(q^{-n})\\&=\frac{q^{-d}+q^{-(n-d)}}{q-1}\left(1+O_{0\le}\left(q^{-\min(d,n-d)}\right)+O\left(q^{-n+\min(d,n-d)}\right)\right)\\&= \frac{q^{-d}+q^{-(n-d)}}{q-1}\exp\left(O\left(q^{-\min(d,n-d)}\right)\right).
\end{align*}

Combining this with \eqref{appr29}, we obtain that

 \begin{equation}
 1-P_n(d,n-d)P_n(n-d,d)=\frac{q^{-d}+q^{-(n-d)}}{q-1}\exp\left(O\left( q^{-\min(d,n-d)}\right)\right).\label{appr3}
 \end{equation}

Furthermore, for $d<n$, we have
\begin{align}
P_n(d,n-d-1)&=\frac{q^{-d}}{q-1}\left(1-q^{-(n-d)}\right)\prod_{i=0}^{n-d-2} \left(1-q^{-n+i}\right)\label{appr1}\\
&=\frac{q^{-d}}{q-1}\exp\left(O\left(q^{-(n-d)}\right)\right)\exp\left(\sum_{i=0}^{n-d+1}O\left(q^{-(n-i)}\right)\right)\nonumber\\&=\frac{q^{-d}}{q-1}\exp\left(O\left(q^{-\min(d,n-d)}\right)\right)\nonumber.
\end{align}

Combining \eqref{appr2}, \eqref{appr3} and~\eqref{appr1}, we obtain~\eqref{approx2} and~\eqref{approx3}.
\end{proof}






Instead of $Q_i$ and $\widehat{X}_i$ it will be more convenient to consider $S_i$ and $\widetilde{X}_i$ defined below. As it turns out later, $Q_i=S_i$ and $\widehat{X}_i=\widetilde{X}_i$ with high probability for all small enough $i$.

Let $S_0=0$, and
\[S_{j+1}=\min\{i>S_j\,:\, X_{2i-2}+X_{2i-1}\neq n\text{ or } X_{2i-1}+X_{2i}\neq n\}.\]

We also define $\widetilde{X}_0=X_0-m$ and for $j>0$,
\[\widetilde{X}_j=\begin{cases}
X_{2S_j}-m&\text{if }(n-X_{2S_j-2}-X_{2S_j-1})+(n-X_{2S_j-1}-X_{2S_j})=1,\\
*&\text{otherwise.}
\end{cases}\]
Here $*$ is a symbol distinct from the elements of $\mathbb{Z}$. Note that if $\widetilde{X}_j\neq *$ and $\widetilde{X}_{j+1}\neq *$, then $|\widetilde{X}_{j+1}-\widetilde{X}_j|=1$.

\begin{lemma}\label{skeletoncoupling}
The total variation distance of $(\widehat{Z}_{j})_{j=0}^{q^{2\alpha n}}$ and $(\widetilde{X}_{j})_{j=0}^{q^{2\alpha n}}$ is at most $O\left(q^{-2\alpha n}\right)$.
\end{lemma}
\begin{proof}
Using \eqref{approx2} and \eqref{approx3}, we see that for all $a_1,a_2\in\mathbb{Z}$ such that $|a_1-a_2|=1$, we have
\[\mathbb{P}(\widetilde{X}_{j+1}=a_2\,|\, \widetilde{X}_{j}=a_1)=\mathbb{P}(\widehat{Z}_{j+1}=a_2\,|\, \widehat{Z}_{j}=a_1)\exp\left(O\left(q^{|a_1|-m}\right)\right).\]

Thus, given a sequence $a=a_0,a_1,\dots,a_\ell$, we see that
\[\mathbb{P}(\widetilde{X}_{j}=a_j\text{ for all }j\le \ell)=\mathbb{P}(\widehat{Z}_{j}=a_j\text{ for all }j\le \ell)\exp\left(\ell(O\left(q^{\max|a_j|-m}\right) \right).\]

Therefore, it follows that
\begin{multline*}\mathbb{P}(\widehat{Z}_{j}=a_j\text{ for all }j\le \ell)-\mathbb{P}(\widetilde{X}_{j}=a_j\text{ for all }j\le \ell)\\\le\mathbb{P}(\widehat{Z}_{j}=a_j\text{ for all }j\le \ell) \min(O\left(\ell q^{\max|a_j|-m}\right),1).
\end{multline*}

Recall the general fact that if $\nu$ and $\mu$ are two probability measures on a discrete set $\mathcal{X}$, then
\[\dTV(\nu,\mu)=\sum_{x\in \mathcal{X}} |\nu(x)-\mu(x)|_+=\sum_{\substack{x\in \mathcal{X}\\\nu(x)>0}} \nu(x)\frac{|\nu(x)-\mu(x)|_+}{\nu(x)}.\]

Thus, the total variation distance of $(\widehat{Z}_{j})_{j=0}^{q^{2\alpha n}}$ and $(\widetilde{X}_{j})_{j=0}^{q^{2\alpha n}}$ is bounded by
\begin{equation}\label{coupbound}
 \mathbb{E} \min\left(O\left( q^{2\alpha n+\max_{j=0}^{q^{2\alpha n}}|\widehat{Z}_{j}|-m}\right),1\right)\le \mathbb{P}\left(\max_{j=0}^{q^{2\alpha n}}|\widehat{Z}_{j}|>2\alpha n\right)+O\left( q^{(4\alpha-0.5) n}\right).\end{equation}

Here
\begin{align}\label{farfrom0bound}
\mathbb{P}\left(\max_{j=0}^{q^{2\alpha n}}|\widehat{Z}_{j}|>2\alpha n\right)&\le \sum_{j=1}^{q^{2\alpha n}} \mathbb{P}\left(|\widehat{Z}_{j-1}|=\lfloor 2\alpha n\rfloor, |\widehat{Z}_{j}|=\lfloor 2\alpha n\rfloor+1 \right)\\&\le \sum_{j=1}^{q^{2\alpha n}} \mathbb{P}\left( |\widehat{Z}_{j}|=\lfloor 2\alpha n\rfloor +1\,\Big|\,|\widehat{Z}_{j-1}|=\lfloor 2\alpha n\rfloor\right)\nonumber \\&=q^{2\alpha n} O\left(q^{-4\alpha n}\right)=O\left(q^{-2\alpha n}\right).\nonumber
\end{align}

Since $4\alpha-0.5<-2\alpha$, combining \eqref{coupbound} and \eqref{farfrom0bound} the statement follows.
\end{proof}
\begin{lemma}\label{geomvsexp}
Let $G$ be a geometric random variable such that
\[\mathbb{P}(G\ge k)=(P_n(d,n-d)P_n(n-d,d))^{k-1}\]
and let $F$ be an exponential random variable such that
\[\mathbb{P}(F\ge t)=\exp\left(-t\left(q^{-(d-m)}+q^{d-m}\right)\right).\]
Then there is a coupling of $G$ and $F$ such that
\[\left|\frac{q^{-m}}{q-1} G-F\right|\le O\left(F q^{-\min(d,n-d)}\right)+O\left(q^{-m}\right).\]

Consequently, provided that $|m-d|\le 2\alpha n$, we have
\[\mathbb{P}\left(\left|\frac{q^{-m}}{q-1} G-F\right|>q^{-3\alpha n}\right)\le \exp\left(-\Omega\left(q^{(0.5-5\alpha)n}\right)\right).\]

\end{lemma}
\begin{proof}
 We construct a coupling as follows: Let $U$ be a uniform random element of $[0,1]$, and let
 \[F=\frac{-\log U}{q^{-(d-m)}+q^{d-m}}.\]
 Furthermore, let 
 \[G=\min\{k\in \mathbb{N}\,:\,(P_n(d,n-d)P_n(n-d,d))^{k}\le U\}.\]
 It is easy to see that this way we get a coupling of $G$ and $F$.

 Also, since
 \[(G-1)\log(P_n(d,n-d)P_n(n-d,d))>\log U\ge G\log(P_n(d,n-d)P_n(n-d,d)),\]
 we have
 \begin{multline*}
     (G-1)\log(P_n(d,n-d)P_n(n-d,d))\\>-F\left(q^{-(d-m)}+q^{d-m}\right)\\\ge G\log(P_n(d,n-d)P_n(n-d,d)).
 \end{multline*}
 Thus by \eqref{approx1},
 \begin{multline*}
  (G-1)\left(\frac{q^{-d}+q^{-(n-d)}}{q-1}+O\left(q^{-2\min(d,n-d)}\right)\right)\\<F\left(q^{-(d-m)}+q^{d-m}\right)\\\le G\left(\frac{q^{-d}+q^{-(n-d)}}{q-1}+O\left(q^{-2\min(d,n-d)}\right)\right).
 \end{multline*}
 Therefore,
 \[\frac{q^{-m}}{q-1}(G-1)\left(1+O\left(q^{-\min(d,n-d)}\right)\right)<F\le \frac{q^{-m}}{q-1}G\left(1+O\left(q^{-\min(d,n-d)}\right)\right).\]

 Thus,
 \[\left|\frac{q^{-m}}{q-1} G-F\right|= O\left(G q^{-m} q^{-\min(d,n-d)}\right)+O\left(q^{-m}\right)\le O\left(F q^{-\min(d,n-d)}\right)+O\left(q^{-m}\right) .\]

 Let $\varepsilon>0$ be sufficiently small. If $|d-m|\le 2\alpha n$, then $\min(d,n-d)\ge (0.5-2\alpha)n$. On the event that $F\le \varepsilon q^{(0.5-5\alpha)n}$, we have
\[\left|\frac{q^{-m}}{q-1} G-F\right|\le O\left(\varepsilon q^{(0.5-5\alpha)n-(0.5-2\alpha)n}+q^{-m}\right)\le q^{-3\alpha n}\]
provided that $\varepsilon$ is sufficiently small and $n$ is large enough.

Thus,
\[\mathbb{P}\left(\left|\frac{q^{-m}}{q-1} G-F\right|>q^{-3\alpha n}\right)\le \mathbb{P}(F\ge \varepsilon q^{(0.5-5\alpha)n})\le \exp\left(-\Omega\left(q^{
(0.5-5\alpha)n}\right)\right).\qedhere\]
\end{proof}

\begin{lemma}\label{timecoupling}
Let $a=a_0,a_1,\dots,a_{q^{2\alpha n}}$ be a sequence such that $|a_i-a_{i-1}|=1$ for all $i$, and $|a_i|\le 2\alpha n$ for all $i$. Let $(S_0',S_1',\dots,S_{q^{2\alpha n}}')$ be distributed the same way as $(S_0,S_1,\dots,S_{q^{2\alpha n}})$ conditioned on the event that $\widetilde{X}_j=a_j$ for all $j$. Let $(T_0',T_1',\dots,T_{q^{2\alpha n}}')$ be distributed the same way as $(T_0,T_1,\dots,T_{q^{2\alpha n}})$ conditioned on the event that $\widehat{Z}_j=a_j$ for all $j$. Then there is a coupling of $(S_0',S_1',\dots,S_{q^{2\alpha n}}')$ and $(T_0',T_1',\dots,T_{q^{2\alpha n}}')$ such that with probability at least $1-O\left(q^{-2\alpha n}\right)$, we have $\left|\frac{q^{-m}}{q-1}S_i'-T_i'\right|\le q^{-n\alpha}$ for all $i$.
\end{lemma}
\begin{proof}
Let $(F_i,G_i)$ be a coupling provided by Lemma~\ref{geomvsexp} with the choice $d=a_i+m$. We can also ensure that $(F_i,G_i)$ ($i=0,1,\dots,q^{2\alpha n}-1$) are independent. Then $S_i'=\sum_{j=0}^{i-1} G_j$ and $T_i'=\sum_{j=0}^{i-1} F_j$ gives us a coupling of $(S_0',S_1',\dots,S_{q^{2\alpha n}}')$ and $(T_0',T_1',\dots,T_{q^{2\alpha n}}')$. On the event that $\left|\frac{q^{-m}}{q-1} G_i-F_i\right|\le q^{-3n\alpha}$ for all $i$, we have $|\frac{q^{-m}}{q-1}S_i'-T_i'|\le q^{-n\alpha}$. The probability of that $\left|\frac{q^{-m}}{q-1} G_i-F_i\right|> q^{-3n\alpha}$ for some $i$ is at most $q^{2\alpha n}\exp\left(-\Omega\left(q^{(0.5-5\alpha)n}\right)\right)=O\left(q^{-2\alpha n}\right)$ with a very rough estimate.
\end{proof}

Combining Lemma~\ref{skeletoncoupling}, \eqref{farfrom0bound} and Lemma~\ref{timecoupling}, we can obtain a coupling of $X$ and $Z$ such that with probability at least $1-O\left(q^{-2\alpha n}\right)$, we have
\begin{enumerate}[(a)]
\item For all $0\le i\le q^{2\alpha n}$, we have $\widetilde{X}_i=\widehat{Z}_i$;\label{eventa}
\item For all $0\le i\le q^{2\alpha n}$, we have $\left|\frac{q^{-m}}{q-1}S_i-T_i\right|\le q^{-\alpha n}$;
\end{enumerate}

Note that on the event \eqref{eventa}, we have 
\[ |n-X_{2i-2}-X_{2i-1}|+|n-X_{2i-1}-X_{2i}|\le 1 \text{ for all $1\le i\le S_{\lfloor q^{2n\alpha}\rfloor}$,}\]
and $Q_i=S_i$ for all $0\le i\le q^{2\alpha n}$, so:
\[\text{For all $0\le i\le q^{2\alpha n}$, we have $\left|\frac{q^{-m}}{q-1}Q_i-T_i\right|\le q^{-\alpha n}$.}\]

Thus, to prove Theorem~\ref{theoremcoupling0} the only thing left to be proved is the following lemma:
\begin{lemma}
Provided that $|Z_0|\le 2\alpha n$, we have
\[\mathbb{P}(T_{\lfloor q^{2\alpha n} \rfloor}<q^{\alpha n})=O\left(q^{-3\alpha n}\right).\]
\end{lemma}
\begin{proof}
 By Lemma~\ref{noexplosion}, the probability of the event above can be bounded by
 \begin{align*}
 \mathbb{P}\left(\text{Pois}(2q^{\alpha n})\ge \frac{\lfloor q^{2\alpha n} \rfloor-|Z_0|}2\right)&\le \mathbb{P}\left(\text{Pois}(2q^{\alpha n})\ge \frac{\lfloor q^{2\alpha n} \rfloor-2\alpha n}2\right) \\&\le \frac{2q^{\alpha n}}{\left(\frac{\lfloor q^{2\alpha n} \rfloor-2\alpha n}2-2q^{\alpha n}\right)^2}\\&=O\left(q^{-3\alpha n}\right)
  \end{align*}
 by Chebyshev's inequality.
\end{proof}

\section{Bounding the number of jumps of $Z_t$}\label{SecZt}

\begin{lemma}\label{noexplosion}
Assume that $Z_0=a$ for some constant $a$. Then the number of jumps up to time $t$ is stochastically dominated by $|a|+2\text{Pois}(2t)$. In particular $Z_t$ has no explosions.

Moreover, if $a\le 0$, then $D_{a,t}$ is stochastically dominated by a Poisson(2t) random variable.
\end{lemma}
\begin{proof}
 Let $A_{a,t}$ be the number of times when $Z$ moves away from $0$ up to time $t$, and let $B_{a,t}$ the number of times when $Z$ moves towards $0$ up to time $t$. Let $U_{a,t}$ be the number upward jumps up to time $t$. We have
 \begin{align*}
  Z_t-Z_0&=U_{a,t}-D_{a,t},\\
  |Z_t|-|Z_0|&=A_{a,t}-B_{a,t},\\
  U_{a,t}+D_{a,t}&=A_{a,t}+B_{a,t}.
 \end{align*}

 From these equations we obtain that the total number of jumps is
 \begin{equation}\label{totalnumberofjumps}A_{a,t}+B_{a,t}=2A_{a,t}+|Z_0|-|Z_t|\le 2A_{a,t}+|a|. \end{equation}
 
 Regardless of our current position, we jump away from $0$ with rate at most $2$. Thus, the number of jumps up to time $t$ that move away from $0$ is stochastically dominated by a $\text{Pois}(2t)$ random variable. Thus, the first statement follows from \eqref{totalnumberofjumps}.

 Moreover,
 \begin{align*}D_{a,t}&=\frac{Z_0-Z_t+A_{a,t}+B_{a,t}}2\\&=A_{a,t}+\frac{Z_0+|Z_0|-Z_t -|Z_t|}2\\&=A_{a,t}+|a|_+-|Z_t|_+\\&\le A_{a,t}+|a|_+.\end{align*}
 Thus, the last statement follows.
\end{proof}

\begin{corollary}\label{corPoissonmoment}
 Assuming that $a\le 0$, we have
 \[\left\|{D}_{a,t}^j\right\|_{\frac{4}3}=O(t^3)\text{ for all $t\ge 1$ and $j=0,1,2,3$}.\]
\end{corollary}
\begin{proof}
    The statement is trivial for $j=0$, so we may assume that $j>0$. It follows from Lemma~\ref{noexplosion} that
    \[\left\|{D}_{a,t}^j\right\|_{\frac{4}3}=\left(\mathbb{E}{D}_{a,t}^\frac{4j}{3}\right)^{\frac{3}4}\le \left(\mathbb{E}|\text{Pois}(2t)|^4\right)^{\frac{3}4}=O(t^3).\qedhere\]
\end{proof}

Choose $Z_0$ such that $\mathbb{E} q^{|Z_0|}<\infty$ and let \[g(t)=\mathbb{E} q^{|Z_t|}.\] 

\begin{lemma}\label{lemmaderivative}
For all $t\ge 0$,
\begin{equation}\label{semicont}
\liminf_{s\to t-} g(s)\ge g(t)\ge \limsup_{s\to t+} g(s)
\end{equation}
and
\[g'_+(t)\le -\frac{g(t)^2}q+1+2q, \]
where $g'_+$ denotes the upper right-hand derivative of $g$, that is,
\[g'_+(t)=\limsup_{s\to t+}\frac{g(s)-g(t)}{s-t}.\]
\end{lemma}
\begin{proof}
Let $a>0$. Assume that $Z_t=a$. If in the time interval $[t,t+\varepsilon]$, we never traverse the directed edges $(a,a+1),(a-1,a)$ and $(-a+1,-a)$, but we do traverse the directed edge $(a,a-1)$, then $|Z_{t+\varepsilon}|< a$. On this event, we have $|q^{|Z_{t+\varepsilon}|}-q^{|Z_t|}|_-\ge q^{a-1}$. The event above has probability at least
\[\left(1-\exp(-\varepsilon q^{a})\right)\exp\left(-\varepsilon\left(2q^{-(a-1)}+q^{-a}\right)\right).\]
Thus,
\[\mathbb{E}(|q^{|Z_{t+\varepsilon}|}-q^{|Z_t|}|_-\quad|\,X_t=a)\ge q^{a-1}(1-\exp(-\varepsilon q^{a}))\exp(-\varepsilon(2q^{-(a-1)}+q^{-a})).\]

It follows that
\[\liminf_{\varepsilon\to 0+} \varepsilon^{-1}\mathbb{E}(|q^{|Z_{t+\varepsilon}|}-q^{|Z_t|}|_-\quad|\,Z_t=a)\ge q^{2a-1}> q^{2|a|-1}-1.\]

By symmetry the inequality above also true for $a<0$, and it trivially true for $a=0$ as well. Thus, by Fatou's lemma,
\begin{equation}\label{negpart}\liminf_{\varepsilon\to 0+}\varepsilon^{-1}\mathbb{E}\left|q^{|Z_{t+\varepsilon}|}-q^{|Z_t|}\right|_-\ge \mathbb{E}q^{2|Z_t|-1}-1\ge \frac{1}q \left(\mathbb{E} q^{|Z_t|}\right)^2-1=\frac{1}q g(t)^2-1,\end{equation}
where the last inequality follows from the Cauchy-Schwartz inequality.

Let $a\ge 0$. Assume that $Z_t=a$. If $|Z_{t+\varepsilon}|=b>a$, then in the time interval $[t,\varepsilon]$, we must traverse all the directed edges $(a,a+1),(a+1,a+2),\dots,(b-1,b)$ or we must traverse all the directed $(-a,-a-1),(-a-1,-a-2),\dots,(-b+1,-b)$. The probability of this event is at most
\[2\prod_{i=a}^{b-1}\varepsilon q^{-i}.\]
Thus,
\begin{align}\label{nemno}
\mathbb{E}(\mathbbm{1}\left(|Z_{t+\varepsilon}|>|a|)q^{|Z_{t+\varepsilon}|}\,|\,Z_t=a\right)\le \sum_{b=a+1}^{\infty} 2q^b\prod_{i=a}^{b-1}\varepsilon q^{-i}\le 2q\varepsilon+2q\sum_{i=2}^{\infty}\varepsilon^i=2q\varepsilon+2q\frac{\varepsilon^2}{1-\varepsilon}
\end{align}
assuming that $\varepsilon<1$. By symmetry, \eqref{nemno} is also true for $a<0$.

Therefore,
\begin{equation}\label{eq232412}\mathbb{E}|q^{|Z_{t+\varepsilon}|}-q^{|Z_t|}|_+\le \mathbb{E}\mathbbm{1}(|Z_{t+\varepsilon}|>|Z_t|)q^{|Z_{t+\varepsilon}|}\le 2q\varepsilon+2q\frac{\varepsilon^2}{1-\varepsilon}\text{ for all }0<\varepsilon<1, \end{equation}
so
\[g(t+\varepsilon)\le g(t)+2q\varepsilon+2q\frac{\varepsilon^2}{1-\varepsilon}\text{ for all }0<\varepsilon<1,\]
which implies \eqref{semicont}.

It follows from \eqref{eq232412} that
\[\limsup_{\varepsilon\to 0}\varepsilon^{-1}\mathbb{E}|q^{|Z_{t+\varepsilon}|}-q^{|Z_t|}|_+\le 2q.\]

Combining this with \eqref{negpart}, we obtain the second part of the lemma.
\end{proof}

The next two lemmas are easy calculus exercises.
\begin{lemma}\label{lemmacalc1}
Let $g$ be a function satisfying \eqref{semicont} such that $g_+'(t)< c$ for all $t\in[a,b)$, then $g(b)-g(a)\le c(b-a)$.
\end{lemma}
\begin{proof}
Let $s=\sup\{t\in [a,b]\,:\,g(t)-g(a)\le c(t-a)\}$. By \eqref{semicont}, we have $g(s)-g(a)\le c(s-a)$. Thus, it is enough to prove that $s=b$. Assume that $s<b$.  By the definition of $s$, we have $g(s+\varepsilon)-g(a)> c(s+\varepsilon-a)$ for all $b-s>\varepsilon>0$. So $g(s+\varepsilon)-g(s)>\varepsilon c$ for all $b-s>\varepsilon>0$, which contradicts $g_+(s)< c$. 
\end{proof}

\begin{lemma}\label{lemmacalc2}
Let $g$ be a function satisfying \eqref{semicont} such that $g_+'(t)< 0$ for all $t$ such that $g(t)= L$. Assume that $g(a)\le L$, then $g(t)\le L$ for all $t>a$.
\end{lemma}
\begin{proof}
Assume that $\{t>a\,:\,g(t)>L\}$ is nonempty. Let $s=\inf\{t>a\,:\,g(t)>L\}$. It follows from \eqref{semicont} that $g(s)=L$. But then by the choice of $s$, we have $g'_+(s)\ge 0$, which is a contradiction.
\end{proof}

\begin{lemma}\label{rexists}
There is a non-increasing function $r:(0,\infty)\to \mathbb{R}$ (not depending on the law of $Z_0$) such that 
\[g(t)\le r(t)\text{ for all $t>0$, provided that $g(0)<\infty$.} \]

\end{lemma}
\begin{proof}
By Lemma~\ref{lemmaderivative}, there is a positive integer $K$ such that if $g(t)\ge K$, then $g_+'(t)\le -\frac{1}{2q} g(t)^2$. 

Let us define
\[r(t)=\max\left(K\,,\,\min\left\{j\in\mathbb{Z}_+\,:\, \sum_{i=j}^\infty i^{-2}<\frac{t}{2q}\,\right\}\right).\]

Since $\sum_{i=1}^\infty i^{-2}<\infty$, $r(t)$ is finite.

Let
\[\tau_i=\inf\{ t\ge 0\,:\, g(t)\le i\}.\]
It follows from Lemma~\ref{lemmacalc1} that if $i\ge K$, then $\tau_{i}-\tau_{i+1}\le \frac{2q}{i^2}$.

Let us choose an integer $j$ such that $j>g(0)$. Then $\tau_j=0$. Thus,

\[\tau_{r(t)}=\sum_{i={r(t)}}^{j-1}(\tau_{i}-\tau_{i+1})\le 2q\sum_{i={r(t)}}^{j-1}i^{-2}\le 2q\sum_{i={r(t)}}^{\infty}i^{-2}<t.\]

Since $r(t)\ge K$, $g_+'(s)<0$ for all $s$ such that $g(s)=r(t)$. Since $\tau_{r(t)}<t$, we have an $a<t$ such that $g(a)\le r(t)$. Thus, Lemma~\ref{lemmacalc2} gives that $g(t)\le r(t)$.
\end{proof}

As a straightforward corollary of Lemma~\ref{rexists}, we get the following lemma.
\begin{lemma}\label{Zcontinuity}
Assume that $g(0)<\infty$. Then for all $s>0$, $\varepsilon>0$, we have
\[\mathbb{P}(Z_t\text{ has a jump in the time interval }[s,s+\varepsilon])\le (r(s)+1)\varepsilon.\]
\end{lemma}

\subsection{The limit of $(D_{a,t},Z_{a,t})$ as $a\to -\infty$ -- The proof of Lemma~\ref{Dinftyexists}}

It is straightforward to see that the set of probability measures on a discrete countable set is complete with respect to the total variation distance. Thus, it is enough to prove that for every $\varepsilon>0$, there is $K$ such that if $b<a<K$, then $\dTV\left((D_{a,t},Z_{a,t})\,,\,(D_{b,t},Z_{b,t})\right)<\varepsilon$.

Let $Z$ and $Z'$ processes with initial states $Z_0=b$, $Z_0'=a$ and transition rate matrix $Q_Z$. We assume that these are independent. Let $D_t$ and $D_t'$ count the number of downward jumps made by the process $Z$ and $Z'$ until time $t$, respectively. Let us define
\[\tau=\min(t,\inf\{s>0\,:\,Z_s=a\}),\]
and let $\nu$ be the law of $\tau$.

Let $B$ be the number of downward jumps of $Z$ until time $\tau$ (including the jump at $\tau$ if that is a downward jump.) Then $(D_t,Z_t)$ has the same distribution as  $(B+D'_{t-\tau},Z'_{t-\tau}+E)$, where $E=\mathbbm{1}(\tau=t)(Z_t-a)$. Therefore, $\left((B+D'_{t-\tau},Z'_{t-\tau}+E)\,,\,(D'_t,Z'_t)\right)$ provides a coupling of $(D_t,Z_t)$ and $(D_t',Z_t')$. 

Thus, for any $0<\delta<t/2$, we have 
\begin{align}\label{dTVbecs1}\dTV\left((D_t,Z_t)\,,\,(D_t',Z_t')\right)&\le \mathbb{P}((B+D'_{t-\tau},Z'_{t-\tau}+E)\neq (D_t',Z_t'))\\&\le \mathbb{P}(B>0)+\mathbb{P}((D'_{t-\tau},Z'_{t-\tau}+E)\neq (D_t',Z_t))\nonumber\\
&\le \mathbb{P}(B>0)+\int_0^t \mathbb{P}((D'_{t-s},Z'_{t-s}+E)\neq (D_t',Z_t')) \nu(ds)\nonumber\\&\le
\mathbb{P}(B>0)+\mathbb{P}(\tau>\delta)+\int_0^\delta \mathbb{P}((D'_{t-s},Z'_{t-s})\neq (D_t',Z_t'))\nu(ds).\nonumber
\end{align}
Note that by Lemma~\ref{Zcontinuity} for any $s\le \delta<t/2$, we have
\[\mathbb{P}((D'_{t-s},Z'_{t-s})\neq (D_t',Z_t'))\le \mathbb{P}(Z'\text{ has a jump in }[t-s,t])\le (r(t-s)+1) s\le (r(t/2)+1)\delta.\]

Thus, combining this with~\eqref{dTVbecs1}, we see that
\[\dTV\left((D_t,Z_t)\,,\,(D_t',Z_t')\right)\le \mathbb{P}(B>0)+\mathbb{P}(\tau>\delta)+(r(t/2)+1)\delta. \]

Let us choose $\delta>0$ such that $(r(t/2)+1)\delta<\frac{\varepsilon}{3}$.

Observe that
\begin{align*}\mathbb{P}(B=0)&\ge \mathbb{P}(\widehat{Z}_i=b+i\text{ for all }1\le i\le a-b)\\&=\prod_{b\le i<a}\frac{q^{-i}}{q^{-i}+q^{i}}\\&\ge \prod_{ i<a}\frac{q^{-i}}{q^{-i}+q^{i}}\\&= \exp\left(-\sum_{i<K}O\left(q^{2i}\right)\right)\ge 1-\frac{\varepsilon}6,\end{align*}
provided that $K$ is small enough.

Next, we have
\begin{align*}\mathbb{P}&(\tau>\delta)\\&\le \mathbb{P}(\widehat{Z}_i\neq b+i\text{ for some }1\le i\le a-b)+\mathbb{P}(\tau>\delta,\widehat{Z}_i=b+i\text{ for all }1\le i\le a-b)\\
&\le\frac{\varepsilon}6+\mathbb{P}\left(\tau>\delta\,\Big|\,\widehat{Z}_i=b+i\text{ for all }1\le i\le a-b\right).
\end{align*}

Let $\tau'=\inf\{s>0\,:\,Z_s=a\}$. Observe that $\tau'\ge \tau$.
Conditioned on the event that $\widehat{Z}_i=b+i\text{ for all }1\le i\le a-b$, $\tau'$ has the same distribution as $F_b+F_{b+1}+\dots+F_{a-1}$, where $F_i$ are independent exponential random variables with parameter $q^i+q^{-i}$. 

Thus,
\[\mathbb{E}(\tau'\,|\,\widehat{Z}_i=b+i\text{ for all }1\le i\le a-b)=\sum_{b\le i<a}\frac{1}{q^i+q^{-i}}\le \sum_{ i<K}\frac{1}{q^i+q^{-i}}\le \frac{\varepsilon\delta}{3}\]
provided that $K$ is small enough.

Combining with Markov's inequality, we see that
\[\mathbb{P}(\tau>\delta)\le \frac{\varepsilon}6+\mathbb{P}\left(\tau>\delta\,\Big|\,\widehat{Z}_i=b+i\text{ for all }1\le i\le a-b\right)\le \frac{\varepsilon}6+\frac{\varepsilon}3.\]

Putting everything together, we get that
\[\dTV\left((D_t,Z_t)\,,\,(D_t',Z_t')\right)\le \mathbb{P}(B>0)+\mathbb{P}(\tau>\delta)+(r(t/2)+1)\delta<\frac{\varepsilon}6+\frac{\varepsilon}6+\frac{\varepsilon}3+\frac{\varepsilon}3\le\varepsilon.\]

\subsection{The proof of Corollary~\ref{marginalconv}}

Let $\varepsilon>0$. By Lemma~\ref{Zcontinuity}, we can choose a small enough $\delta>0$ such that $Z_t$ has no jumps in $\cup_{i=1}^k [t_i-\delta,t_i+\delta]$ with probability at least $1-\varepsilon$. Assume that $n$ is large enough such that $|Z_0|<2\alpha n$, $q^{-\alpha n}<\delta$, $q^{\alpha n}>t_k$ and the event in Theorem~\ref{theoremcoupling} has probability at least $1-\varepsilon$.

Then with probability $1-2\varepsilon$, the event in Theorem~\ref{theoremcoupling} occurs, moreover $Z_t$ has no jumps in~$\cup_{i=1}^k [t_i-\delta,t_i+\delta]$. It is straightforward to see that on this event $Y_{t_i}=Z_{t_i}$ for all $1\le i\le k$.




\section{The corank at criticality}\label{Seccrit}
\subsection{The proof of part~\eqref{thmdimkerC2kpart2} of Theorem~\ref{thmdimkerC2k}}
\begin{lemma}\label{Lemma34}
 Given any $c$, let $\mathcal{A}$ be the event that
 \begin{itemize}
  \item $X_{2i-2}=c$, and
  \item $X_{2i-2}+ X_{2i-1}\neq n$ or $X_{2i-1}+ X_{2i}\neq n$.
 \end{itemize}
  
 Then
 \[\mathbb{P}\left(X_{2i-1}+ X_{2i}\neq n\text{ or }X_{2i-2}+X_{2i-1}\le n-\alpha n \,|\, \mathcal{A}\right)\le O\left(q^{2c-n}\right). \]
\end{lemma}
\begin{proof}

 First, it is straightforward to see that
 \[P_n(d,r)=\mathbbm{1}(d+r\le n)O\left(q^{-(n-d-r)(n-r)}\right).
\]

Thus,
\begin{align*}1-P_n(d,n-d)&=\sum_{r=0}^{n-d-1}P_n(d,r)\\&=\sum_{r=0}^{n-d-1} O\left(q^{-(n-d-r)(n-r)}\right)\\&=\sum_{r=0}^{n-d-1} O\left(q^{-(n-r)}\right)\\&=O\left(q^{-d}\right).
\end{align*}

Therefore,
\begin{align*}\sum_{r=0}^{n-d} P_n(d,r)(1-P_n(r,n-r))&=\sum_{r=0}^{n-d} O\left(q^{-(n-d-r)(n-r)}\right) O\left(q^{-r}\right)\\&=\sum_{r=0}^{n-d-2} O\left(q^{-2(n-r)}\right) O\left(q^{-r}\right)+O\left(q^{-n}\right)+O\left(q^{-(n-d)}\right)\\&=O\left(q^{-(n-d)}\right)+O\left(q^{-n}\right) \sum_{r=0}^{n-d-2} O\left(q^{-(n-r)}\right)\\&=O\left(q^{-(n-d)}\right).\end{align*}

Furthermore,
\[\sum_{r=0}^{n-d-\alpha n} P_n(d,r)=\sum_{r=0}^{n-d-\alpha n}O\left(q^{-(n-d-r)(n-r)}\right)=O\left(q^{-\alpha n (\alpha n+d)}\right)=O\left(q^{-(n-d)}\right).\]

Thus, 

\begin{equation}\label{eq847463}
 \mathbb{P}(X_{2i-1}+ X_{2i}\neq n\text{ or }X_{2i-2}+X_{2i-1}\le n-\alpha n \,|\, X_{2i-2}=c)=O\left(q^{-(n-c)}\right).
\end{equation}

Note that
\begin{align*}\mathbb{P}(X_{2i-2}&+ X_{2i-1}\neq n\text{ or }X_{2i-1}+ X_{2i}\neq n\,|\, X_{2i-2}=c)\\&\ge P_n(c,n-c-1)P_n(n-c-1,c+1)\\&\ge \Omega\left(q^{-c}\right).\end{align*}

Combining this with \eqref{eq847463}, the statement follows.
\end{proof}

\begin{lemma}\label{gettingcloseto0}
Assume that $X_0=0$. Let $\tau$ be the smallest $i$ such that $X_{2i}\ge m-2\alpha n.$ Then with probability $1-O\left(q^{-\alpha n}\right)$, we have $\tau\le q^{m-\alpha n}$, $ \alpha n\le m-X_{2\tau}\le 2\alpha n$ and \[X_{2i-1}+X_{2i}=n\text{ for all $1\le i\le \tau$}.\]
\end{lemma}
\begin{proof}
Let $\tau_j$ be the smallest $i$ such that $X_{2i}= j$. Furthermore, let
\[\tau_j'=\min\{i>\tau_j\,:\, X_{2i-2}+X_{2i-1}\neq n \text{ or }X_{2i-1}+X_{2i}\neq n\}.\]

Moreover, let $\mathcal{A}_j$ be the event that $X_{2\tau_j'-1}+X_{2\tau_j'}= n$ and $X_{2\tau_j'-2}+X_{2\tau_j'-1}\ge n- \alpha n$. By Lemma~\ref{Lemma34}, the probability that the events $\mathcal{A}_0,\mathcal{A}_1,\dots, \mathcal{A}_{m-2\alpha n}$ all occur is at least
\[1-\sum_{j=0}^{m-\alpha n}(1-\mathbb{P}(\mathcal{A}_j))=1-\sum_{j=0}^{m-2\alpha n}O\left(q^{2j-n}\right)=1-O\left(q^{-4\alpha n}\right).\]

On this event, we have that $ \alpha n\le m-X_\tau\le 2\alpha n$, 
\[X_{2i-1}+X_{2i}=n\text{ for all $1\le i\le \tau$},\]
and
\[\tau\le \sum_{j=0}^{n-\alpha n} (\tau_j'-\tau_j).\]
By \eqref{appr3}, we have
\[\mathbb{E}\sum_{j=0}^{m-2\alpha n} (\tau_j'-\tau_j)= \sum_{j=0}^{m-2\alpha n} (1-P_n(j,n-j)P_n(n-j,j))^{-1}=\sum_{j=0}^{m-2\alpha n} O\left(q^j\right)=O\left(q^{m-2\alpha n}\right).\]

Thus, by Markov's inequality,
\[\mathbb{P}\left(\tau>q^{m-\alpha n}\right)\le O\left(q^{- \alpha n}\right).\qedhere\]
\end{proof}

By almost the same proof, we can obtain the following statement.
\begin{lemma}\label{gettingcloseto0b}
Assume that $|X_0-m|>2\alpha n$. Let $\tau$ be the smallest $i$ such that $|X_{2i}-m|\le 2\alpha n.$ Then \[\mathbb{P}(\tau\le q^{m-\alpha n})=1-O\left(q^{-\alpha n}\right).\]
\end{lemma}
\begin{proof}
    When $X_0<m$, the proof of Lemma~\ref{gettingcloseto0} works. This is also true for the case $X_0>m$, but in place of Lemma~\ref{Lemma34}, we need to use the following statement, which can be proved the same way Lemma~\ref{Lemma34}:
    
     Given any $c$, let $\mathcal{A}$ be the event that
 \begin{itemize}
  \item $X_{2i-2}=c$, and
  \item $X_{2i-2}+ X_{2i-1}\neq n$ or $X_{2i-1}+ X_{2i}\neq n$.
 \end{itemize}
  
 Then
 \[\mathbb{P}\left(X_{2i-1}+ X_{2i}\le n-\alpha n\text{ or }X_{2i-2}+X_{2i-1}\neq n \,|\, \mathcal{A}\right)\le O\left(q^{2(n-c)-n}\right). \qedhere\]
    
\end{proof}

For $0\le a\le n$, let $X_i$ be a Markov chain with transition matrix $P_n$ such that $X_0=a$, and let
\[D'_{a,j}=D'_{a,j,n}=\sum_{i=1}^j (n-X_{2i-1}-X_{2i}).\]

This definition is motivated by the fact that if $X_i=\rang(C_i)-\rang(C_{i-1})$, then
\[\dim\ker C_{2k}=\sum_{i=1}^k (n-X_{2i-1}-X_{2i}).
\]
Thus, we have the following lemma:
\begin{lemma}\label{dimkervsD}
The distribution of $\dim\ker C_{2k}$ is the same as that of $D'_{0,k}$.
 
\end{lemma}

On the event that $|n-X_{2i-2}-X_{2i-1}|+|n-X_{2i-1}-X_{2i}|\le 1$, we have
\[n-X_{2i-1}-X_{2i}=\begin{cases}
1&\text{if }X_{2i}=X_{2i-2}-1,\\
0&\text{otherwise.}
\end{cases}
\]

Thus, on the event that $|n-X_{2i-2}-X_{2i-1}|+|n-X_{2i-1}-X_{2i}|\le 1$ for all $1\le i\le k$, we have
\begin{equation}\label{Dpasdown}D_{a,k}'=\sum_{i=1}^k \mathbbm{1}(X_{2i}=X_{2i-2}-1).\end{equation}
The equation above shows the significance of the event \eqref{couplingeventc} in Theorem~\ref{theoremcoupling0}.

The proof of the next lemma is a straightforward combination of Theorem~\ref{theoremcoupling0}, Lemma~\ref{Zcontinuity} and equation \eqref{Dpasdown}.
\begin{lemma}\label{DDprime}
 Assume that $|a-m|\le 2\alpha n$. Given $j$, let $t=\frac{q^{-m}}{q-1} j$, and assume that $t\le q^{\alpha n}$. Let $X_0=a$, $Z_0=a-m$ and consider the coupling provided by Theorem~\ref{theoremcoupling0}. Furthermore, assume that the event in Theorem~\ref{theoremcoupling0} occurs and $Z_{a-m,t}$ has no jumps in the interval $\left[t-q^{-\alpha n},t+q^{-\alpha n}\right]$. Then \[(D'_{a,j},X_{a,2j}-m)=(D_{a-m,t},Z_{a-m,t}).\]

 Furthermore, let $t_0>0$, and assume that $t_0\le t-q^{-\alpha n}$. Then the probability of the event above is at least $1-O((r(t_0)+1)q^{-\alpha n})$.
\end{lemma}

Part \eqref{thmdimkerC2kpart2} of Theorem~\ref{thmdimkerC2k} follows by combining Lemma~\ref{dimkervsD} with the next lemma.
\begin{lemma}\label{criticalconvergence}
Let $k_n$ be a sequence of positive integers and let $t_n=\frac{q^{-m}}{q-1}k_n$. 

Assume that $\lim_{n\to\infty} t_n=t$, where $0<t<\infty$. Then $(D_{0,k_n,n}',X^{(n)}_{2k_n}-m)$ converges in total variation distance to $(D_{-\infty,t},Z_{-\infty,t}).$ (Here $X^{(n)}$ has the initial state $X^{(n)}_0=0$.) 

\end{lemma}
\begin{proof}
Assume that the high probability event $\mathcal{A}$ in Lemma~\ref{gettingcloseto0} occurs. Then, let $(D_{0,k_n}'',X_{2k_n}''-m)$ be distributed as $(D'_{0,k_n},X_{2k_n}-m)$ conditioned on the event $\mathcal{A}$. Since $\mathbb{P}(\mathcal{A})$ tends to $1$,
\begin{equation}\label{firststep}\dTV((D_{0,k_n}',X_{2k_n}-m)\,,\,(D_{0,k_n}'',X_{2k_n}''-m))<\varepsilon\end{equation}
for all large enough $n$.

Let $\tau$ be defined as in Lemma~\ref{gettingcloseto0}. Then
\begin{multline*}\mathbb{P}(D_{0,k_n}''=d, X_{2k}''-m=b)\\=\sum_{\substack{\ell,a\\ \alpha n \le m-a\le 2\alpha n,\quad \ell\le q^{m-\alpha n}}} \mathbb{P}(\tau=\ell,X_{2\tau}=a\,|\,\mathcal{A}) \mathbb{P}(D'_{a,k_n-\ell}=d,X_{a,2(k_n-\ell)}-m=b).\end{multline*}

Let $\varepsilon>0$.

Let us choose $a,\ell$ such that $ \alpha n \le m-a\le 2\alpha n$ and $\ell\le q^{m-\alpha n}$. Let $s=\frac{q^{-m}}{q-1} \ell$.

Provided that $n$ is large enough for all choices of $a,\ell$ as above, we have \begin{align*}
 \dTV((D'_{a,k_n-\ell},X_{a,2(k_n-\ell)}-m)\,,\,(D_{a-m,t_n-s},Z_{a-m,t_n-s}))&<\varepsilon&\text{(by Lemma~\ref{DDprime})},\\
 \dTV((D_{a-m,t_n-s},Z_{a-m,t_n-s})\,,\,(D_{a-m,t},Z_{a-m,t}))&<\varepsilon&\text{(by Lemma~\ref{Zcontinuity} }\\&\text{and the fact}&t_n-s-t=o(1)\, ),\\
 \dTV((D_{a-m,t},Z_{a-m,t})\,,\,(D_{-\infty,t},Z_{-\infty,t}))&<\varepsilon&\text{(by Lemma~\ref{Dinftyexists})}.
\end{align*}

Thus,
\begin{multline*}\dTV((D_{0,k_n}'',X''_{2k_n}-m)\,,\,(D_{-\infty,t},Z_{-\infty,t}))\\\le \sum_{\substack{\ell,a\\ \alpha n \le m-a\le 2 \alpha n, \\\ell\le q^{m-\alpha n}}} \mathbb{P}(\tau=\ell,X_{2\tau}=a\,|\,\mathcal{A}) \dTV((D'_{a,k_n-\ell},X_{a,2(k_n-\ell)}-m)\,,\,(D_{-
\infty,t},Z_{-\infty,t}))\\\le 3\varepsilon.
\end{multline*}

Combining this with~\eqref{firststep}, it follows that
\[\dTV((D_{0,k_n}',X_{2k_n}-m),(D_{-\infty,t},Z_{-\infty,t}))\le 4\varepsilon.\]

Thus, the statement follows. 
\end{proof}

\subsection{The proof of parts~\eqref{theoremtruncatedpart1} and \eqref{theoremtruncatedpart2} of Theorem~\ref{theoremtruncated}}\label{seccriticaltrunc}

Let $\widetilde{C}_i$ obtained from $C_i$ by deleting the first $n/2$ rows of $C_{i}$. Let \[X_i^*=\rang \widetilde{C}_{i+1}-\rang \widetilde{C}_{i}.\] Then
$X_0^*,X_1^*,X_2^*,\dots$ is a Markov chain with transition matrix $P_n$.

For $-m\le u\le m$, $W_{n,u}$ be a random variable independent of $X_0^*,X_1^*,\dots$ such that
\begin{multline*}\mathbb{P}(W_{n,u}=r)\\=\mathbb{P}(\text{A uniform random $m\times(m-u)$ matrix over $\mathbb{F}_q$ has rank }\min(m,m-u)-r).\end{multline*}

By Remark~\ref{remarkOthern}, we see that 
$\rang(\widehat{C}_{2k})-\rang(\widetilde{C}_{2k+1})$ conditioned on $X_{2k}^*=m+u$ has the same distribution as the rank of an $m\times (m-u)$ matrix. In other words, $\rang(\widehat{C}_{2k})-\rang(\widetilde{C}_{2k+1})$ has the same distribution as $m-|u|_+-W_{n,u}$.

Thus, $\rang(\widehat{C}_{2k})$ has the same distribution as
\[\sum_{i=0}^{2k}X_i^*+m-|X_{2k}^*-m|_+-W_{n,X_{2k}^*-m}.\]

Thus, $\dim\ker(\widehat{C}_{2k})$ has the same distribution as ${D}_k^*$, where
\begin{align*}{D}_{k}^*&=(k+1)n-\left(\sum_{i=0}^{2k}X_i^*+m-|X_{2k}^*-m|_+-W_{n,X_{2k}^*-m}\right)\\&=m-X_0^*+\sum_{i=0}^{k-1}(n-X_{2i+1}^*-X_{2i+2}^*)+|X_{2k}^*-m|_+ +W_{n,X_{2k}^*-m}.
\end{align*}

Conditioned on the event that $X_0^*=m$, the joint distribution of $X_0^*,X_1^*,\dots$ and ${D}_k^*$ is the same as the joint distribution of $X_0,X_1,\dots$ and \[\widetilde{D}_k=D_{m,k}'+|X_{2k}-m|_++W_{n,X_{2k}-m},\] where $X_0=m$, and $X_0,X_1,\dots$ a Markov chain with transition matrix $P_n$, and as before
\[D_{m,k}'=\sum_{i=1}^{k} (n-X_{2i-1}-X_{2i}).\]

It is straightforward to see that
\[\lim_{n\to\infty}\mathbb{P}(X_0^*=m)=1.\]

Thus,
\begin{equation}\label{DstD}
\lim_{n\to\infty}\dTV(D_{k_n}^*,\widetilde{D}_{k_n})=0.
\end{equation}

Relying on the formula given in Remark~\ref{remarkOthern}, it is straightforward to see that for any fixed $u$, $W_{n,u}$ converge to $J_{|u|}$ in total variation distance as $n\to \infty$. Combining this with Lemma~\ref{DDprime}, we see that $\widetilde{D}_{k_n}$ converges to $L_t$. By \eqref{DstD}, $D_{k_n}^*$ converges to $L_t$ as well. Thus,  parts~\eqref{theoremtruncatedpart1} and \eqref{theoremtruncatedpart2} of Theorem~\ref{theoremtruncated} follow.

\section{Gaussian limits}\label{SecGauss}

\subsection{The proof of part~\eqref{theoremtruncatedpart3} of Theorem~\ref{theoremtruncated}}




We will use the notations of Section~\ref{seccriticaltrunc}.

\begin{lemma}\label{Wsmall}
 For all choices for $u$ and $n$, we have
 \[\mathbb{E}W_{n,u}\le 2.\]
\end{lemma}
\begin{proof}
First, assume that $u\ge 0$. Let $B$ be a uniform random $m\times (m-u)$ matrix over $\mathbb{F}_q$. Then
\[\mathbb{E}W_{n,u}=\mathbb{E}\dim\ker B\le \mathbb{E}|\ker B|=1+(q^{m-u}-1)q^{-m}\le 2.\]

Next, assume that $u< 0$. Let $B$ be a uniform random $(m-u)\times m$ matrix over $\mathbb{F}_q$. Then
\[\mathbb{E}W_{n,u}=\mathbb{E}\dim\ker B\le \mathbb{E}|\ker B|=1+(q^{m}-1)q^{-(m-u)}\le 2.\qedhere\]
\end{proof}

Combining Lemma~\ref{Wsmall} with Lemma~\ref{tight}, we see that as $n$ goes through the even integers, the sequence of random variables
\[\widetilde{D}_{k_n}-D_{m,k_k}'=|X_{2k}-m|_++W_{n,X_{2k}-m}\]
is tight. If we combine this with \eqref{DstD}, we see that it is enough to prove that there is a constant $\sigma>0$ such that if $\lim_{n\to\infty}t_n=\infty$, then
\begin{equation}\label{enoughnormal}\frac{D_{m,k_n}'-\mu_n t_n}{\sigma \sqrt{t_n}}
\end{equation}
converges in distribution to a standard normal variable.

In the next few lemmas, our aim is to find the analogues of Lemma~\ref{noexplosion} and Corollary~\ref{corPoissonmoment} for the Markov chains $X^{(n)}$. 

First, let us introduce some terminology:
\begin{itemize}
\item We say that we have a simple step at $i$, if $X_{2i}+X_{2i+1}=n$ or $X_{2i+1}+X_{2i+2}=n$.

\item We say that we have a simple step away from $m$ at $i$, if we have a simple step at $i$ and $m\le X_{2i}< X_{2i+2}$ or $m\ge X_{2i}> X_{2i+2}$.

\item We say that we have a simple step towards $m$ at $i$, if we have a simple step at $i$ and $m\le X_{2i+2}< X_{2i}$ or $m\ge X_{2i+2}> X_{2i}$.

\item We say that we have an exceptional step at $i$, if we do not have a simple step at $i$ or $X_{2i}<m<X_{2i+2}$ or $X_{2i+2}<m<X_{2i}$.
\end{itemize}
Let
\begin{align*}
 A_i&=\mathbbm{1}(\text{we have a simple step away from $m$ at $i$}) |X_{2i+2}-X_{2i}|,\\ 
 B_i&=\mathbbm{1}(\text{we have a simple step towards $m$ at $i$}) |X_{2i+2}-X_{2i}|,\\
 C_i&=\mathbbm{1}(\text{we have an exceptional step at $i$})\left( |n-X_{2i+1}-X_{2i}|+|n-X_{2i+2}-X_{2i+1}|\right).
\end{align*}

It is straightforward to see that
\begin{equation}\label{ineqc1}|X_{2k}-m|+\sum_{i=0}^{k-1} B_i\le |X_0-m|+\sum_{i=0}^{k-1} (A_i+C_i).
\end{equation}

Also,

\begin{equation}\label{eqc1}X_0+\sum_{i=0}^{k-1}(n-X_{2i}-X_{2i+1})=X_{2k}+\sum_{i=0}^{k-1}(n-X_{2i+1}-X_{2i+2}).
\end{equation}

Thus,
\begin{align*}
2\sum_{i=0}^{k-1}(n-X_{2i+1}-X_{2i+2})&= X_0-X_{2k}+\sum_{i=0}^{k-1}(n-X_{2i+1}-X_{2i+2})+\sum_{i=0}^{k-1}(n-X_{2i}-X_{2i+1})\\&=X_0-X_{2k}+\sum_{i=0}^{k-1}(A_i+B_i+C_i)\\&\le |X_0-m|+X_0-|X_{2k}-m|-X_{2k}+2\sum_{i=0}^{k-1}(A_i+C_i)\\&=2|X_0-m|_+-2|X_{2k}-m|_++2\sum_{i=0}^{k-1}(A_i+C_i),
\end{align*}
where the first equality follows from \eqref{eqc1}, and the inequality follows from \eqref{ineqc1}.
Thus, if $a\le m$, then
\begin{equation}\label{DAC}D'_{a,k}\le \sum_{i=0}^{k-1}(A_i+C_i).
\end{equation}


\begin{lemma}\label{Hexsits}
There is a random variable $H=H^{(n)}$ such that
\[\mathbb{E}|H|^j=O\left(q^{-m}\right)\text{ for all }j=1,2,3,4,\]
and
\[\sum_{i=0}^{k-1} (A_i+C_i) \text{ is stochastically dominated by }\sum_{i=0}^{k-1} H_i,\]
where $H_0,H_1,\dots,H_{k-1}$ are i.i.d. copies of $H$.

\end{lemma}

\begin{proof}
First assume that $d> m$. Then for $j\ge 1$, we have
\[\mathbb{P}(A_i=j\,|\,X_{2i}=d)=P_n(d,n-d-j)P_n(n-d-j,d+j)=O(q^{-j(d+j)})=O(q^{-jm}).\]
Next assuming that $d<m$, for all $j\ge 1$, we have
\[\mathbb{P}(A_i=j\,|\,X_{2i}=d)=P_n(d,n-d)P_n(n-d-j,d-j)=O(q^{-j(n-d+j)})=O(q^{-jm}).\]
Finally,
\[\mathbb{P}(A_i=j\,|\,X_{2i}=m)=P_n(m,m)P_n(m,m-j)+P_n(m,m-j)P(m-j,m+j)=O(q^{-jm}).\]

Thus, 
\[\mathbb{P}(A_i=j)=O(q^{-jm}).\]

For any $j\ge 2$, we have
\begin{align*}
\mathbb{P}&(C_i=j\text{ and we do not have a simple step at $i$ }|\,X_{2i}=d)\\&=\sum_{\ell=1}^{j-1} P_n(d,n-d-\ell)P_n(n-d-\ell,d+2\ell-j)\\&=\sum_{\ell=1}^{j-1} O(q^{-\ell(d+\ell)-(j-\ell)(n-d-2\ell+j)}).
\end{align*}
Combining this with the fact that
\begin{align*}
 -\ell(d+\ell)-(j-\ell)(n-d-2\ell+j)&=-\ell d-(j-\ell)(n-d)-\frac{3(2\ell-j)^2}4-\frac{j^2}4\\&\le -n-\frac{3(2\ell-j)^2}4-\frac{j^2}4,
\end{align*}
we obtain that
\[\mathbb{P}(C_i=j\text{ and we do not have a simple step at $i$ }|\,X_{2i}=d)=O\left(q^{-n-\frac{j^2}4}\right).\]

If $C_i=j$ and we have a simple step at $i$, then we must have $X_{2i}=d$ for some $d$ such that $0<|d-m|<j$. Assume that $d$ is like that. First, consider the case that $d>m$. Then $n-d+j>m$, so
\begin{align*}
\mathbb{P}(C_i=j\text{ and we have a simple step at $i$ }|\,X_{2i}=d)&= P_n(d,n-d)P_n(n-d,d-j)\\&=O(q^{-j(n-d+j)})\\&=O(q^{-jm}). 
\end{align*}

Similarly, if $d<m$, then $d+j>m$, so
\begin{align*}
\mathbb{P}(C_i=j\text{ and we have a simple step at $i$ }|\,X_{2i}=d)&= P_n(d,n-d-j)P_n(n-d,d+j)\\&=O(q^{-j(d+j)})\\&=O(q^{-jm}). 
\end{align*}

Thus, there is a non-negative integer valued random variable $H$ such that 
\[\mathbb{P}(H=j)=\begin{cases}
 O\left(q^{-m}\right)&\text{for $j=1$},\\
 O\left(q^{-n-\frac{j^2}4}+q^{-jm}\right)&\text{for $j\ge 2$},
\end{cases}
\]
and
\[\sum_{i=0}^{k-1} (A_i+C_i) \text{ is stochastically dominated by }\sum_{i=0}^{k-1} H_i,\]
where $H_0,H_1,\dots,H_{k-1}$ are i.i.d. copies of $H$.

It is straightforward to check that $H$ satisfies $\mathbb{E}|H|^j=O\left(q^{-m}\right)$ for all $j=1,2,3,4$.
\end{proof}

\begin{lemma}\label{finitePoissonmoment}
Assume that $k\ge q^m$, and $a\le m$. Then

\[\left\|(D'_{a,k})^{j}\right\|_{\frac{4}3}=O(k^3q^{-3m})\text{ for all $j=0,1,2,3$}.
\]
\end{lemma}
\begin{proof} 
Let $H_i$ be as in Lemma~\ref{Hexsits}. Expanding the product
\[\mathbb{E}\left(\sum_{i=0}^{k-1} H_i\right)^4,\]
we obtain 
\begin{itemize}
 \item $k$ terms of the the form $\mathbb{E}H_i^4$. The total contribution of these terms is $O(kq^{-m})$.
 \item $O(k^2)$ terms of the form $\mathbb{E}H_i^3 H_j$. The total contribution of these terms is $O(k^2q^{-2m})$.
 \item Furthermore, the terms of the form $H_i^2H_j^2,\,H_i^2H_jH_h,\, H_iH_jH_hH_g$ have total contribution $O(k^2q^{-2m})$, $O(k^3q^{-3m})$, $O(k^4q^{-4m})$, respectively.
\end{itemize}
Thus, using the assumption that $k\ge q^m$, we see that

\begin{equation}\label{eqH4}\mathbb{E}\left(\sum_{i=0}^{k-1} H_i\right)^4\le O(k^4 q^{-4m}).\end{equation}

The statement of the lemma is clear for $j=0$. Thus, assume that $j>0$. Combining \eqref{DAC}, Lemma~\ref{Hexsits} and \eqref{eqH4}, we obtain
\[\left\|(D'_{a,k})^{j}\right\|_{\frac{4}3}=\left(\mathbb{E} (D'_{a,k})^{\frac{4j}3}\right)^{\frac{3}4}\le \left( \mathbb{E}\left(\sum_{i=0}^{k-1} H_i\right)^4\right)^{\frac{3}4}\le O(k^3 q^{-3m}).\qedhere\]
\end{proof}

Assume that $Z_0=0$. Let $h_{\infty,0}=0$, and for $i>0$, let
\[h_{\infty,i}=\min\{j>h_{\infty,i-1}\,:\, \widehat{Z}_j=0\}.\]
Then, let 
\[U_{\infty,i}=T_{h_{\infty,i}}-T_{h_{\infty,i-1}},\]
that is, $U_{\infty,i}$ is the length of the $i$th excursion from $0$.

Let $X_i=X_i^{(n)}$ be a Markov chain with transition matrix $P_n$ such that $X_0=m$. Recall that $R_i, Q_i$ and $\widehat{X}^{(n)}_i$ were defined in \eqref{Ridef}, \eqref{Qidef} and \eqref{hatXdef}, respectively.

Let $h_{n,0}=0$, and for $i>0$, let
\[h_{n,i}=\min\{j>h_{n,i-1}\,:\, \widehat{X}_{j}^{(n)}=0\}.\]
Then let 
\[U_{n,i}=R_{h_{n,i}}-R_{h_{n,i-1}},\]

We also define $(\Delta D)_{\infty,i}$ as the number of down steps made by $Z_t$ during its $i$th excursion from~$0$, more formally,
\[(\Delta D)_{\infty,i}=D_{0,T_{h_{\infty,i}}}-D_{0,T_{h_{\infty,i-1}}},\]
and

\[(\Delta D)_{n,i}={D}'_{m,Q_{h_{n,i}},n}-{D}'_{m,Q_{h_{n,i-1}},n}.\]

We also set
\[U_{\infty}=U_{\infty,1},\quad U_{n}=U_{n,1},\quad (\Delta D)_{\infty}=(\Delta D)_{\infty,1},\quad (\Delta D)_{n}=(\Delta D)_{n,1}.\]

Finally, let
\[V_{\infty,i}= (\Delta D)_{\infty,i}-\mu U_{\infty,i}, \quad\text{ where }\mu=\frac{\mathbb{E} (\Delta D)_\infty}{\mathbb{E} U_\infty},\]
and
\[V_{n,i}= (\Delta D)_{n,i}-\mu_n U_{n,i}, \quad\text{ where }\mu_n=\frac{\mathbb{E} (\Delta D)_n}{\mathbb{E} U_n}.\]

Note that at this point it is unclear that the above definition of $\mu_n={\mathbb{E} (\Delta D)_n}/{\mathbb{E} U_n}$ indeed agrees with the one given in \eqref{mundef}. However, once we prove that \eqref{enoughnormal} holds with the choice of $\mu_n={\mathbb{E} (\Delta D)_n}/{\mathbb{E} U_n}$, it follows that the two definitions of $\mu_n$ must be the same. Indeed, it is clear from the Markov chain Central Limit Theorem and Lemma~\ref{lemmaMarkov} that \eqref{enoughnormal} can be only true for the constant $\mu_n$ given in \eqref{mundef}. 

Using the notations $V_{\infty}=V_{\infty,1}$ and $V_{n}=V_{n,1}$, we see that
\begin{itemize}
 \item The sequence $U_{\infty,1},U_{\infty,2},\dots$ consists of i.i.d copies of $U_\infty$.
 
 \item The sequence $U_{n,1},U_{n,2},\dots$ consists of i.i.d copies of $U_n$.

 \item The sequence $V_{\infty,1},V_{\infty,2},\dots$ consists of i.i.d copies of $V_\infty$.
 
 \item The sequence $V_{n,1},V_{n,2},\dots$ consists of i.i.d copies of $V_n$.
\end{itemize}

Note that it follows from the definitions of $V_\infty$ and $V_n$ that
\[\mathbb{E}V_{\infty}=0\quad\text{ and }\quad\mathbb{E}V_n=0\text{ for all }n.\]
\begin{lemma}\label{Utail}
There are constants $b>1$ and $c<1$ such that
\[\mathbb{P}(U_\infty>bj)\le c^j \quad\text{ and }\quad \mathbb{P}(U_n>bj)\le c^j\text{ for all positive integers $j$ and $n$}.\]

Moreover, for
\[\tau_\infty=\inf\{t>0\,:\, Z_t=0\}\quad\text{ and }\quad\tau_n=\inf\{i>0\,:\, X^{(n)}_{2i}=m\},\]
we have
\[\mathbb{P}(\tau_\infty>bj)\le c^j\quad \text{ and }\quad\mathbb{P}\left(\frac{q^{-m}}{q-1}\tau_n>bj\right)\le c^j\]
 for all positive integers $j$ and $n$ regardless the distribution of $Z_t$ and $X_0^{(n)}$.
\end{lemma}
\begin{proof}
 First, we prove the statement for $U_\infty$ and $\tau_\infty$. It is enough to prove that there is $b$ and $c<1$ such that regardless the value of $Z_0$, we reach $0$ in time $b$ with at least probability $1-c$. If $Z_0=0$ it is meant that we first need to leave $0$. 

 We may assume that $Z_0\neq 0$. Because with positive probability we leave $0$ in constant time. By symmetry, we may assume that $Z_0=a>0$. Consider the event that
 \[\widehat{Z}_i=a-i\text{ for all }i=1,2,\dots,a.\]
 The probability of this event is
 \[\prod_{i=0}^{a-1} \exp\left(-O\left(q^{-2(a-i)}\right)\right)\ge \exp\left(-O\left(\sum_{i=0}^\infty q^{-2i}\right)\right)\ge c_0,\]
 for some positive constant $c_0$ not depending on $a$. Conditioned on this event, the expected time to reach $0$ is
 \[\sum_{i=0}^{a-1} O(q^{-(a-i)})\le O\left(\sum_{i=0}^\infty q^{-i}\right)<b_0,\]
 for some positive constant $b_0$ not depending on $a$. 
 Thus, by Markov's inequality, conditioned on the event above, we reach $0$ in time at most $2b_0$ with probability at least $\frac{1}2$. Thus, with probability at least $\frac{c_0}2$, we reach $0$ is time $2b_0$.

 To prove the statement on $U_n$ and $\tau_n$, it is enough to prove that there is $b$ and $c$ such that regardless the value of $X^{(n)}_0$, 
 \[\mathbb{P}(\min\{i\,:\,X^{(n)}_{2i}=m\}<b(q-1)q^m)\ge 1-c.\]
 Using Lemma~\ref{gettingcloseto0b} and Theorem~\ref{theoremcoupling}, this can be reduced to the already established statement about~$Z_t$.
\end{proof}

\begin{lemma}\label{Ljapunovcond}
For all $j,h\in\{0,1,2,3\}$, we have
\[\lim_{n\to\infty} \mathbb{E}U_n^j(\Delta D)_n^h =\mathbb{E}U_\infty^j(\Delta D)_\infty^h<\infty.\]

Consequently,

\[\lim_{n\to \infty} \Var(V_n)=\Var(V_\infty),\]
and there an $M<\infty$ such that
\[\Var(U_\infty)<M,\quad \mathbb{E}|V_\infty|^3<M\quad\text{and}\quad\Var(U_n)<M,\quad\mathbb{E}|V_n|^3<M\text{ for all }n.\]
\end{lemma}
\begin{proof}

 Combining Lemma~\ref{Utail}, H\"{o}lder's inequality and Lemma~\ref{corPoissonmoment}, we obtain that 
\begin{align*}\mathbb{E}U_\infty^j(\Delta D)_\infty^h&\le \sum_{\ell=0}^\infty ((\ell+1)b)^j\mathbb{E}\mathbbm{1}(\ell b<U_\infty\le (\ell+1) b ) D_{0,(\ell+1) b}^h\\&\le \sum_{\ell=0}^\infty ((\ell+1)b)^j\left(\mathbb{P}(\ell b<U_\infty\le (\ell+1) b )\right)^{
\frac{1}4}\| D_{0,(\ell+1) b}^h\|_{\frac{4}3}\\&=\sum_{\ell=0}^\infty O\left(\ell^3\right) O\left(c^{\frac{\ell}4}\right)O(\ell^3)<\infty. 
\end{align*}

 Let us consider the coupling provided by Theorem~\ref{theoremcoupling0} with the choice of $a=0$. The lemma will follow once we prove that under this coupling, we have
 \begin{equation}\label{ucoup}
     \lim_{n\to\infty}\mathbb{E}|U_\infty^j (\Delta D)_\infty^h-U_n^j (\Delta D)_n^h|=0.\end{equation}
 
 Let $\mathcal{A}$ be the intersection of the high probability event given in Theorem~\ref{theoremcoupling0} and the event that $U_\infty\le q^{0.1\alpha n}$. By Theorem~\ref{theoremcoupling0} and Lemma~\ref{Utail} this event has probability at least $1-O(q^{-2n\alpha})$. Let $\mathcal{A}^c$ be the complement of the event $\mathcal{A}$. Then
\begin{multline}\label{triang}\mathbb{E}|U_\infty^j (\Delta D)_\infty^h-U_n^j (\Delta D)_n^h|\\\le \mathbb{E}\mathbbm{1}(\mathcal{A})|U_\infty^j (\Delta D)_\infty^h-U_n^j (\Delta D)_n^h|+\mathbb{E}\mathbbm{1}(\mathcal{A}^c)U_\infty^j (\Delta D)_\infty^h+\mathbb{E}\mathbbm{1}(\mathcal{A}^c)U_n^j (\Delta D)_n^h.\end{multline}

  On the  event $\mathcal{A}$, we have \begin{equation}\label{Ubecs}|U_\infty^j-U_n^j|=O(q^{-0.8\alpha n}).\end{equation}
  Indeed, this is trivial for $j=0$, and for $j>0$, on the event $\mathcal{A}$, we have 
  \[|U_\infty^j-U_n^j|\le |U_\infty-U_n| O(|U|_\infty^{j-1})=O(q^{-\alpha n})O(q^{0.2\alpha n})=O(q^{-0.8\alpha n}).\] 
  Note that on the event $\mathcal{A}$, we have $(\Delta D)_\infty=(\Delta D)_n$. Combining this with \eqref{Ubecs} and Lemma~\ref{corPoissonmoment}, we see that 
  \begin{align}\label{tri1}\mathbb{E}\mathbbm{1}(\mathcal{A})|U_\infty^j (\Delta D)_\infty^h-U_n^j (\Delta D)_n^h|&=\mathbb{E}\mathbbm{1}(\mathcal{A})|U_\infty^j -U_n^j |(\Delta D)_\infty^h\\&\le O(q^{-0.8\alpha n})\mathbb{E}D_{0,q^{0.1\alpha n}}^h\nonumber\\&=O(q^{-0.5\alpha n}).\nonumber\end{align}

 Combining H\"{o}lder's inequality and Lemma~\ref{corPoissonmoment}, we see that  \begin{align}\label{Uih1}\mathbb{E}\mathbbm{1}\left(\mathcal{A}^c\wedge \left(U_\infty\le q^{0.05\alpha n} \right)\right) U_\infty^j (\Delta D)_\infty^h&\le q^{0.15\alpha n} \mathbb{E}\mathbbm{1}\left(\mathcal{A}^c\right)  D_{0,q^{0.05\alpha n}}^h\\&\le q^{0.15\alpha n}\left(\mathbb{P}\left(\mathcal{A}^c\right)\right)^{\frac{1}4}\|D_{0,q^{0.05\alpha n}}^h\|_{\frac{4}3}\nonumber\\&=O(q^{-0.2\alpha n}).\nonumber\end{align}

Let $\ell_{0}(n)=\lfloor b^{-1} q^{0.05\alpha n}\rfloor$, where $b$ is the constant provided by Lemma~\ref{Utail}. Then  combining Lemma~\ref{Utail}, H\"{o}lder's inequality and Lemma~\ref{corPoissonmoment}, we obtain
\begin{align}\label{Uih2}\mathbb{E}\mathbbm{1}&\left(\mathcal{A}^c\wedge \left(U_\infty> q^{0.05\alpha n} \right)\right) U_\infty^j (\Delta D)_\infty^h\\&\le \sum_{\ell=\ell_0(n)}^\infty \mathbb{E}\mathbbm{1}(\ell b<U_\infty\le (\ell+1) b )  U_\infty^j (\Delta D)_\infty^h\nonumber\\&\le \sum_{\ell=\ell_0(n)}^\infty  ((\ell+1)b)^j\left(\mathbb{P}(\ell b<U_\infty\le (\ell+1) b )\right)^{
\frac{1}4}\| D_{0,(\ell+1) b}^h\|_{\frac{4}3}\nonumber\\&=\sum_{\ell=\ell_0(n)}^\infty O\left(\ell^3\right) O\left(c^{\frac{\ell}4}\right)O(\ell^3)=o(1)\nonumber\end{align}
by the fact that $\ell_0(n)$ tends to infinity.

Combining \eqref{Uih1} and \eqref{Uih2}, we see that
\begin{equation}\label{tri2}
\lim_{n\to\infty}\mathbb{E}\mathbbm{1}\left(\mathcal{A}^c\right) U_\infty^j (\Delta D)_\infty^h=0.
\end{equation}

A similar argument gives that
\begin{equation}\label{tri3}
\lim_{n\to\infty}\mathbb{E}\mathbbm{1}\left(\mathcal{A}^c\right) U_n^j (\Delta D)_n^h=0,
\end{equation}
the only difference is that we need to rely on Lemma~\ref{finitePoissonmoment} instead of Corollary~\ref{corPoissonmoment} .

Combining \eqref{triang}, \eqref{tri1}, \eqref{tri2} and \eqref{tri3}, we obtain \eqref{ucoup}.
 \end{proof}

We set
\[N_n=\left\lfloor\frac{t_n}{\mathbb{E} U_n}\right\rfloor.\]

Note that $N_n\to \infty$. Combining Lemma~\ref{Ljapunovcond} with the Lyapunov Central Limit Theorem, we obtain the following lemma.
\begin{lemma}\label{lemmaVnormal}
The random variables
\[\frac{1}{\sqrt{N_n \Var(V_n)}}\sum_{i=1}^{N_n} V_{n,i}\]
converges in distribution to a standard normal random variable as $n\to\infty$.
\end{lemma}

Let 
\[N_n'=\min\left\{i\,:\,\sum_{j=1}^i U_{n,j}\ge t_n\right\}.\]

Since $\sup \Var(U_n)<\infty$ by Lemma~\ref{Ljapunovcond}, it follows easily from Chebyshev's in equality that $\frac{N_n'}{N_n}$ converges to $1$ in probability. Thus, there is a sequence $\varepsilon_n$ such that $\lim_{n\to\infty} \varepsilon_n=0$ and \begin{equation}\label{NpinN}\lim_{n\to\infty}\mathbb{P}(N_n'\in [(1-\varepsilon_n)N_n,(1+\varepsilon_n)N_n])=1.\end{equation}

Note that on the event that $N_n'\le N_n$, we have

\begin{multline*}\left|\frac{1}{\sqrt{N_n \Var(V_n)}}\left({D}'_{m,k_n}-\mu_n t_n\right)- \frac{1}{\sqrt{N_n \Var(V_n)}}\sum_{i=1}^{N_n} V_{n,i}\right|\\\le \frac{1}{\sqrt{N_n \Var(V_n)}} \left(\left|\sum_{i=N_n'}^{N_n} V_{n,i}\right| +(\Delta D)_{n,N_n'}+\mu_n U_{n,N_n'}\right),
\end{multline*}
and on the event $N_n'>N_n$, we have
\begin{multline*}\left|\frac{1}{\sqrt{N_n \Var(V_n)}}\left({D}'_{m,k_n}-\mu_n t_n\right)- \frac{1}{\sqrt{N_n \Var(V_n)}}\sum_{i=1}^{N_n} V_{n,i}\right|\\\le \frac{1}{\sqrt{N_n \Var(V_n)}} \left(\left|\sum_{i=N_n+1}^{N_n'} V_{n,i}\right| +(\Delta D)_{n,N_n'}+\mu_n U_{n,N_n'}\right).
\end{multline*}

Therefore, on the even that $N_n'\in [(1-\varepsilon_n)N_n,(1+\varepsilon_n)N_n]$, we have
\begin{align}\label{NNp}&\left|\frac{1}{\sqrt{N_n \Var(V_n)}}\left({D}'_{m,k_n}-\mu_n t_n\right)- \frac{1}{\sqrt{N_n \Var(V_n)}}\sum_{i=1}^{N_n} V_{n,i}\right|\\&\quad\le \frac{1}{\sqrt{N_n \Var(V_n)}} \left(\max_{(1-\varepsilon_n)N_n\le j\le N_n}\left|\sum_{i=j}^{N_n} V_{n,i}\right|+\max_{N_n<j\le (1+\varepsilon_n)N_n }\left|\sum_{i=N_n+1}^{j} V_{n,i}\right|\right)\nonumber\\& \qquad+\frac{1}{\sqrt{N_n \Var(V_n)}}\max_{(1-\varepsilon_n)N_n\le j \le(1+\varepsilon_n)N_n}\left((\Delta D)_{n,j}+\mu_n U_{n,j}\right)\nonumber.
\end{align}

It follows from Kolmogorov's maximal inequality that 
\begin{equation}\label{to01}\frac{1}{\sqrt{N_n \Var(V_n)}} \left(\max_{(1-\varepsilon_n)N_n\le j\le N_n}\left|\sum_{i=j}^{N_n} V_{n,i}\right|+\max_{N_n<j\le (1+\varepsilon_n)N_n }\left|\sum_{i=N_n+1}^{j} V_{n,i}\right|\right)\to 0\end{equation}
in probability.

Moreover, by Lemma~\ref{Ljapunovcond}, we have $\sup \mathbb{E}|(\Delta D)_{n,1}+\mu_n U_{n,1}|^2<\infty$. Thus,
\begin{align*}\mathbb{P}&\left(\max_{(1-\varepsilon_n)N_n\le j \le(1+\varepsilon_n)N_n}\left((\Delta D)_{n,j}+\mu_n U_{n,j}\right)\ge \varepsilon_n^{0.4} \sqrt{N_n}\right)\\&\le O(\varepsilon_n N_n) \mathbb{P}\left((\Delta D)_{n,1}+\mu_n U_{n,1}\ge \varepsilon_n^{0.4} \sqrt{N_n}\right)\\&\le O(\varepsilon_n N_n) \frac{\mathbb{E}|(\Delta D)_{n,1}+\mu_n U_{n,1}|^2}{\varepsilon_n^{0.8} N_n}\\&=O(\varepsilon_n^{0.2})=o(1).\end{align*}

Therefore,
\begin{equation}\label{to02}
    \frac{1}{\sqrt{N_n \Var(V_n)}}\max_{(1-\varepsilon_n)N_n\le j \le(1+\varepsilon_n)N_n}\left((\Delta D)_{n,j}+\mu_n U_{n,j}\right)\text{ converge to $0$ in probability.}
    \end{equation}

Thus, combining \eqref{NpinN}, \eqref{NNp}, \eqref{to01}, \eqref{to02} and  Lemma~\ref{lemmaVnormal}, we obtain that
\begin{equation}\label{almostfullGauss}\frac{1}{\sqrt{N_n \Var(V_n)}}\left({D}'_{m,k_n}-\mu_n t_n\right)\end{equation}
converges in distribution to a standard normal random variable as $n\to\infty$.

Finally, using Lemma~\ref{Ljapunovcond}, we obtain that
\[\lim_{n\to\infty}\frac{N_n \Var(V_n)}{t_n}=\lim_{n\to\infty}\frac{\Var(V_n)}{\mathbb{E} U_n}=\sigma^2,\]
where
\[\sigma^2=\frac{\Var(V_\infty)}{\mathbb{E} U_\infty}.\]

Thus, \eqref{enoughnormal} follows from \eqref{almostfullGauss}.

\subsection{The asymptotic value of $\mu_n$ -- The proof of \eqref{munas}}\label{SecMunas}

For $i\ge 0$, we have
\begin{align*}
\frac{\pi_n(m+i)}{\pi_n(m)}&=q^{-i^2}\prod_{j=0}^{i-1}(1-q^{-m+j})^2(1-q^{-m-j-1})^{-1}\\
&=q^{-i^2}\exp\left(-\sum_{j=0}^{i-1}\left(2q^{-m+j}+O\left(q^{-2(m-j)}\right)-q^{-m-j-1}+O\left(q^{-2(m+j+1)}\right)\right)\right)\\
&=q^{-i^2}\exp\left(-\sum_{j=0}^{i-1}\left(2q^{-m+j}-q^{-m-j-1}\right)+O\left(q^{-2(m-i)}\right)\right)\\&=
q^{-i^2}\left(1-\sum_{j=0}^{i-1}\left(2q^{-m+j}-q^{-m-j-1}\right)+O\left(q^{-2(m-i)}\right)\right).
\end{align*}

A similar argument gives that for $i\ge 0$, we have
\[\frac{\pi_n(m-i)}{\pi_n(m)}=q^{-i^2}\left(1-\sum_{j=0}^{i-1}\left(q^{-m+j}-2q^{-m-j-1}\right)+O\left(q^{-2(m-i)}\right)\right),\]
and so
\begin{align*}\frac{\pi_n(m-i)-\pi_n(m+i)}{\pi_n(m)}&=q^{-i^2}\left(\sum_{j=0}^{i-1}\left(q^{-m+j}+q^{-m-j-1}\right)+O\left(q^{-2(m-i)}\right)\right)\\&=q^{-i^2-m}\frac{q^i-q^{-i}}{q-1}+O\left(q^{-i^2+2i-n}\right).
\end{align*}

By \eqref{mundef}, we have
\begin{align*}
\frac{\mu_n}{\pi_n(m)}&=\frac{(q-1)q^m}{\pi_n(m)} \left(n-2\sum_{i=0}^{n} \pi_n(i)i\right)\\&=2(q-1)q^m\sum_{i=1}^m \left(\frac{\pi_n(m-i)-\pi_n(m+i)}{\pi_n(m)}\right)i\\&
=2(q-1)q^m\sum_{i=1}^m \left(q^{-i^2-m}\frac{q^i-q^{-i}}{q-1}+O(q^{-i^2+2i-n})\right)i\\&=
2\sum_{i=1}^m q^{-i^2}(q^i-q^{-i})i +O(q^{-m})\\&
=2\sum_{i=1}^\infty q^{-i(i-1)}+O(q^{-m}).
\end{align*}

The above argument also gives that for $i\ge 1$, we have
\[\frac{\pi_n(m-i)+\pi_n(m+i)}{\pi_n(m)}=2q^{-i^2}+O(q^{-i^2-m+i}).\]

Thus,
\begin{align*}\frac{1}{\pi_n(m)}&=1+\sum_{i=1}^m\frac{\pi_n(m-i)+\pi_n(m+i)}{\pi_n(m)}\\&=1+2\sum_{i=1}^m q^{-i^2} +O(q^{-m})\\&=1+2\sum_{i=1}^\infty q^{-i^2} +O(q^{-m}).\end{align*}

Therefore,
\[\mu_n=\frac{2 \sum_{i=1}^{\infty} q^{-i(i-1)}}{1+2\sum_{i=1}^\infty q^{-i^2}}+O(q^{-m}).\]

\subsection{The proof of part~\eqref{thmdimkerC2kpart3} of Theorem~\ref{thmdimkerC2k}}

By Lemma~\ref{dimkervsD}, we need to prove that
\begin{equation}\label{GaussasD}
 \frac{D'_{0,k_n,n}-\mu_n t_n}{\sigma \sqrt{t_n}}
\end{equation}
converges to a standard normal variable.

Let $X_i$ be a Markov chain with transition matrix $P_n$ and initial condition $X_0=0$. Let 
\[\tau=\min(\tau_0,k_n),\text{ where }\tau_0=\min\{i\,:\,X_{2i}=m\}.\]

Then $D'_{0,k_n}$ has the same distribution as
\[D'_{0,\tau}+D'_{m,k_n-\tau},\]
where $D'_{m,k_n-\tau}$ is independent of $(X_i)$. Let us choose $\varepsilon_n$ such that \[\lim_{n\to\infty}\varepsilon_n\sqrt{t_n} =0\text{ and }\lim_{n\to\infty}\varepsilon_n t_n=\infty.\]

Combining the condition $\lim_{n\to\infty}\varepsilon_n t_n=\infty$ with Lemma~\ref{Utail}, we obtain that \[\lim_{n\to\infty} \mathbb{P}(\tau\le \varepsilon_n k_n)=1.\]

Then by \eqref{enoughnormal}, conditioned on the event that $\tau\le \varepsilon_n k_n$, we see that
\[\frac{D'_{m,k_n-\tau}-\mu_n\frac{q^{-m}}{q-1}(k_n-\tau)}{\sigma\sqrt{t_n}}\]
converges in distribution to a standard normal variable.

Thus, to obtain the theorem it is enough to prove that
\[\frac{D'_{0,\tau}}{\sqrt{t_n}}\quad\text{ and }\quad\frac{\mu_n\frac{q^{-m}}{q-1}\tau}{\sqrt{t_n}}\quad\text{both converge to $0$ in probability.}\]

The first statement will follow once we prove that $\mathbb{E}D'_{0,\tau}<K$ for some constant $K$ not depending on $n$. Indeed, combining Lemma~\ref{Utail}, Lemma~\ref{finitePoissonmoment} and H\"{o}lder's inequality, we see that 
\begin{align*}
 \mathbb{E}D'_{0,\tau}&\le \sum_{\ell=0}^\infty \mathbb{E}\mathbbm{1}\left(\ell b<\frac{q^{-m}}{q-1}\tau_0\le (\ell+1)b\right) D'_{0,(\ell+1) b (q-1)q^m}
 \\ & \le\sum_{\ell=0}^\infty \left(\mathbb{P}\left(\ell b<\frac{q^{-m}}{q-1}\tau_0\le (\ell+1)b\right)\right)^{\frac{1}{4}}\left \|D'_{0,(\ell+1) b (q-1)q^m}\right\|_{\frac{4}3}\\& \le\sum_{\ell=0}^\infty O\left(c^{\frac{\ell}4}\ell^3\right)<K
\end{align*}
for some $K$ independent of $n$.

Also on the high probability event $\tau<\varepsilon_n k_n$, we have
\[\frac{\mu_n\frac{q^{-m}}{q-1}\tau}{\sqrt{t_n}}\le \mu_n\varepsilon_n\sqrt{t_n}=o(1).\]

\section{The localization/delocalization phase transition}\label{SectionLocalization}

\subsection{The proof of part \eqref{thmdimkerC2kpart1} of Theorem~\ref{thmdimkerC2k}}
Let
\[\mathcal{C}=\{v\in \mathbb{F}_q^{kn}\,:\,\supp(v)\text{ is a nonempty interval }\{a,a+1,\dots,b\}\,\}.\]

Note that if $v\in \mathbb{F}_q^{kn}$ is a vector such that $\supp(v)=\{a,a+1,\dots,b\}$. Then $C_{2k}v$ is a uniform random element of
\[\{u\in \mathbb{F}_q^{(k+1)n}\,:\,\supp(u)\subset\{a,a+1,\dots,b+1\}\,\}.\]
Thus,
\[\mathbb{P}(C_{2k}v=0)=q^{-(b-a+2)n}.\]

Furthermore, the number of vectors $v\in \mathbb{F}_q^{kn}$ such that $\supp(v)=\{a,a+1,\dots,b\}$ \break is $(q^n-1)^{b-a+1}\le q^{(b-a+1)n}$.

Therefore,
\begin{align*}
 \sum_{v\in \mathcal{C}}\mathbb{P}(C_{2k}v=0)&=\sum_{1\le a\le b\le k} \sum_{\substack{v\in \mathbb{F}_q^{kn}\\\supp(v)=\{a,a+1,\dots,b\}}} \mathbb{P}(C_{2k}v=0)\\&\le \sum_{1\le a\le b\le k} q^{(b-a+1)n} q^{-(b-a+2)n}\\&\le k^2 q^{-n}.
\end{align*}

Note that if $k=k_n$ is such that $kq^{-n/2}$ tends to $0$ as $n$ tend to $\infty$, then by Markov's inequality with probability tending to $1$, we have $\mathcal{C}\cap \ker C_{2k}$ is empty. The statement follows by using the following lemma:
\begin{lemma}\label{Lemma44}
 If $\ker C_{2k}$ contains a nonzero vector $v$, then it also contains a vector $v'\in \mathcal{C}$. 
\end{lemma}
\begin{proof}
 Let $a=\min \supp(v)$, and let $c=\min\{c>a\,:\, c\notin \supp(v)\}$. Let us define $v'$ as
 \[v_i'=\begin{cases}
 v_i&\text{if $i<c$},\\
 0&\text{if $i\ge c$}.
 \end{cases}\]

It is straightforward to see that $v'\in\mathcal{C}$ and  $C_{2k}v=0$ is only possible if $C_{2k}v'=0$.
\end{proof}
\subsection{The delocalized phase -- The proof of part \eqref{partdeloc} of Theorem~\ref{thmlocdeloc}}

Let
\[\mathcal{C}_n=\{v\in \mathbb{F}_q^{(k_n+1)n}\,:\,\supp(v)\text{ is a nonempty interval }\{a,a+1,\dots,b\}\,\}.\]

The following lemma can be proved the same way as Lemma~\ref{Lemma44}.

\begin{lemma}
 If $\ker \widehat{C}_{2k}$ contains a nonzero vector $v$, then it is also contains a vector $v'\in \mathcal{C}_n$ such that $\supp(v')\subseteq \supp(v)$. 
\end{lemma}

As a consequence of this lemma, we see that if there is a $v\in \ker \widehat{C}_{2k_n}$ such that $\supp(v)\neq \{1,2,\dots,k_n+1\}$, then there is a $w\in \ker \widehat{C}_{2k_n}$ such that $\supp(w)=\{a,a+1,\dots,b\}$ where $a\neq 1$ or $b\neq k_n+1$.

Thus, it is enough to prove that
\[\lim_{n\to\infty} \mathbb{E}|\{w\in \ker \widehat{C}_{2k_n}\,:\,\supp(w)=\{a,\dots,b\}\text{ where $a\neq 1$ or $b\neq k_n+1$}\}|=0.\]

Observe that 
\begin{align}\label{deloceq1}
\mathbb{E}&|\{w\in \ker \widehat{C}_{2k_n}\,:\,\supp(w)=\{a,\dots,b\}\text{ where $a\neq 1$ or $b\neq k_n+1$}\}|
\\&=\sum_{b=1}^{k_n} \sum_{\substack{w\in \mathbb{F}_q^{(k_n+1)n}\\\supp(w)=\{1,\dots,b\}}} \mathbb{P}(w\in \ker\widehat{C}_{2k_n})\nonumber\\&\quad+\sum_{a=2}^{k_n+1} \sum_{\substack{w\in \mathbb{F}_q^{(k_n+1)n}\\\supp(w)=\{a,\dots,k_n+1\}}} \mathbb{P}(w\in \ker\widehat{C}_{2k_n})\nonumber\\&\quad+\sum_{2\le a\le b\le k_n} \sum_{\substack{w\in \mathbb{F}_q^{(k_n+1)n}\nonumber\\\supp(w)=\{a,\dots,b\}}} \mathbb{P}(w\in \ker\widehat{C}_{2k_n}).\nonumber
\end{align}

Note that
\[|\{w\in \mathbb{F}_q^{(k_n+1)n} \,:\,\supp(w)=\{a,\dots,b\}\}|=(q^n-1)^{b-a+1}\le q^{(b-a+1)n}.\]
It is straightforward to see that for any vector $w$ such that $\supp(w)=\{1,\dots,b\}$ where $b\le k_n$, we have
\[\mathbb{P}(w\in \ker \widehat{C}_{2k_n} )=q^{-bn-m}.\]
Thus, for any $b\le k_n$, we have
\[\sum_{\substack{w\in \mathbb{F}_q^{(k_n+1)n}\\\supp(w)=\{1,\dots,b\}}} \mathbb{P}(w\in \ker\widehat{C}_{2k_n})\le q^{bn} q^{-bn-m}=q^{-m}. \]

Therefore,
\begin{equation}\label{deloceq2}\sum_{b=1}^{k_n}\sum_{\substack{w\in \mathbb{F}_q^{(k_n+1)n}\\\supp(w)=\{1,\dots,b\}}} \mathbb{P}(w\in \ker\widehat{C}_{2k_n})\le q^{-m} k_n. \end{equation}
A similar argument gives that
\begin{equation}\label{deloceq3}\sum_{a=2}^{k_n+1} \sum_{\substack{w\in \mathbb{F}_q^{(k_n+1)n}\\\supp(w)=\{a,\dots,k_n+1\}}} \mathbb{P}(w\in \ker\widehat{C}_{2k_n})\le q^{-m}k_n.\end{equation}

It is straightforward to see that for any vector $w$ such that $\supp(w)=\{a,\dots,b\}$ where $2\le a\le b\le k_n$, we have
\[\mathbb{P}(w\in \widehat{C}_{2k_n} )=q^{-(b-a+1)n-2m}.\]

Thus, for any $2\le a\le b\le k_n$, we have
\[\sum_{\substack{w\in \mathbb{F}_q^{(k_n+1)n}\\\supp(w)=\{a,\dots,b\}}} \mathbb{P}(w\in \ker\widehat{C}_{2k_n})\le q^{(b-a+1)n} q^{-(b-a+1)n-2m}=q^{-2m}. \]

Therefore,
\begin{equation}\label{deloceq4}\sum_{2\le a\le b\le k_n} \sum_{\substack{w\in \mathbb{F}_q^{(k_n+1)n}\\\supp(w)=\{a,\dots,b\}}} \mathbb{P}(w\in \ker\widehat{C}_{2k_n})\le q^{-2m}k_n^2.\end{equation}

Combining \eqref{deloceq1}, \eqref{deloceq2}, \eqref{deloceq3} and \eqref{deloceq4}, we see that
\begin{multline*}\mathbb{E}|\{w\in \ker \widehat{C}_{2k_n}\,:\,\supp(w)=\{a,\dots,b\}\text{ where $a\neq 1$ or $b\neq k_n+1$}\}\le 2q^{-m}k_n+\left(q^{-m}k_n\right)^2.\end{multline*}
Part~\eqref{partdeloc} of Theorem~\ref{thmlocdeloc} follows from the assumption that $\lim_{n\to\infty} q^{-m}k_n=0$.



\subsection{The localized phase -- The proof of part~\eqref{partloc} of Theorem~\ref{thmlocdeloc}}

\begin{lemma}\label{lemmavandown}
There is a $C>0$ depending only on $q$ such that for every $0<\varepsilon\le 1$, we have
\[\mathbb{P}(D_{-\infty,\varepsilon}>0)\ge C\varepsilon^2.\]
\end{lemma}
\begin{proof}
 By the definition of $D_{-\infty,\varepsilon}$ given in Lemma~\ref{Dinftyexists}, it is enough to prove that there is a $C>0$ such that for all $0<\varepsilon\le 1$ assuming that $a$ is small enough, we have
 \[\mathbb{P}(D_{a,\varepsilon}>0)\ge C\varepsilon^2.\]

 Choose $K=\lfloor\log_q \varepsilon\rfloor-2$. Since $\varepsilon\le 1$, we see that $K<0$. Assume that $a<K$. Consider the event that
 \begin{enumerate}[(a)]
  \item \label{eeventa}$\widehat{Z}_i=a+i$ for all $0\le i\le K-a$;
  \item \label{eeventb}$\widehat{Z}_{K-a+1}=K-1$;
  \item \label{eeventc}$T_{K-a+1}\le \varepsilon$.
 \end{enumerate}
 On this event, $D_{a,\varepsilon}>0$. Thus, it is enough to provide a lower bound on the probability of this event.

 The probability that both \eqref{eeventa} and \eqref{eeventb} occur is
 \[\frac{q^{K}}{q^K+q^{-K}}\prod_{i=a}^{K-1} \frac{q^{-i}}{q^i+q^{-i}}\ge \frac{q^{2K}}2 \prod_{i=-\infty}^0 \frac{q^{-i}}{q^i+q^{-i}}>2C \varepsilon^2,\]
 for some $C>0$ depending only on $q$.

 Conditioned on the events \eqref{eeventa} and \eqref{eeventb}, we have
 \[\mathbb{E}T_{K-a+1}=\sum_{i=a}^{K} \frac{1}{q^{i}+q^{-i}}\le \sum_{i=-\infty}^K q^{i}\le 2q^K\le \frac{2}{q^2}\varepsilon\le \frac{\varepsilon}2.\]

 Thus, by Markov's inequality, we have that conditioned on the events \eqref{eeventa} and \eqref{eeventb}, $T_{K-a+1}\le \varepsilon$ with probability at least $\frac{1}2$. Thus, with probability $C\varepsilon^2$, the events \eqref{eeventa}, \eqref{eeventb} and \eqref{eeventc} all occur.
\end{proof}

\begin{corollary}\label{cor48}
 Let $C$ be the constant provided by Lemma~\ref{lemmavandown}. Let $0<\varepsilon\le 1$. Assume that $\limsup_{n\to\infty}\frac{q^{-m}}{q-1} k_n\ge \varepsilon$. Then,
 \[\limsup_{n\to\infty}\mathbb{P}(\dim \ker C_{2k_n}>0)\ge C\varepsilon^2.\]
\end{corollary}
\begin{proof}
 Consider a subsequence $n_1<n_2<n_3<\cdots$ of the positive integers such that \break $\liminf_{i\to\infty}\frac{q^{-n_i/2}}{q-1} k_{n_i}\ge \varepsilon$. 
 
 Let $k_{n_i}'=\min\left(k_{n_i},\lfloor q^{n_i/2}(q-1)\varepsilon\rfloor\right)$. Then $k_{n_i}'\le k_{n_i}$, and $\lim_{i\to\infty}\frac{q^{-n_i/2}}{q-1} k_{n_i}'= \varepsilon$.
 
  Then combining Theorem~\ref{thmdimkerC2k} and Lemma~\ref{lemmavandown}, it follows that
 \[\limsup_{n\to\infty}\mathbb{P}(\dim \ker C_{2k_n'}>0)\ge C\varepsilon^2.\]
 The statement follows from the observation that if $k_n'\le k_n$, then
 \[\mathbb{P}(\dim \ker C^{(n)}_{2k_n'}>0)\le \mathbb{P}(\dim \ker C^{(n)}_{2k_n}>0).\qedhere\]
\end{proof}

Given a vector $v\in\mathbb{F}_q^{(k+1)n}$, the event $\widehat{C}_{2k}v=0$ depends only on the entries of $\widehat{C}_{2k}$ which are contained in a column indexed by an element of the set
\[\cup_{i\in \supp(v)}[n(i-1)+1,ni]\}.\]

Combining this with the facts the columns of $\widehat{C}_{2k_n}$ are independent, and the intervals \break $[(i-1)\bar{k}_n+1,i\bar{k}_n]$ $(i=1,2,\dots,L)$ are pairwise disjoint, we see that the events that $\ker\widehat{C}_{2k_n}\cap V_i\neq \emptyset$ ($i=1,2,\dots,L$) are independent. Moreover,
\[\mathbb{P}(\ker\widehat{C}_{2k_n}\cap V_i\neq \emptyset)\ge \mathbb{P}(\dim \ker C^{(n)}_{2\bar{k}_n}>0) \]
for all $i=1,2,\dots,L$. Thus, part~\eqref{partloc} of Theorem~\ref{thmlocdeloc} follows from Corollary~\ref{cor48}.

\section{The proof of Theorem~\ref{thmdimkerC2k-1}}
\subsection{The connection between block bidiagonal matrices and matrix products -- The proof of part~\eqref{thmdimkerC2k-1part1} of Theorem~\ref{thmdimkerC2k-1}}\label{secproductvsbidiagonal}
Let $A'=\left({A}_{i,j}'\right)_{i,j=1}^\infty$ be a random infinite block matrix over the finite field $\mathbb{F}_q$ such that all the blocks are $n\times n$ matrices, $A'$ is a block lower bidiagonal matrix, that is, $A_{i,j}'=0$ whenever $j \notin \{i,i-1\}$. The blocks $A_{1,1}', A_{2,2}',A_{3,3}',\dots$ are chosen independently uniformly at random from the set of all $n\times n$ matrices over $\mathbb{F}_q$. Moreover, $A_{i,i-1}'=I$ for all $i>1$.

For $k\ge 1$, let $C_{2k-1}'$ be the submatrix of $A'$ determined by the first $kn$ rows and $kn$ columns. Furthermore, let $C_{2k}'$ be the submatrix of $A'$ determined by the first $(k+1)n$ rows and $kn$ columns. 

As it was proved in \cite[Equation (1)]{meszaros2024universal}, we have\footnote{Equation (1) in \cite{meszaros2024universal} is about the cokernel of matrices over $\mathbb{Z}$, but by the same proof one can also obtain the equality \eqref{productasbidiagonal}.}
\begin{equation}\label{productasbidiagonal}\dim\ker C_{2k-1}'=\dim\ker A_{1,1}'A_{2,2}'\cdots A_{k,k}'.
\end{equation}

Let $X_0'=0$ and for $i>0$, let
\[X_i'=\rang(C_i')-\rang(C_{i-1}'),\]
where $\rang(C_0')$ is defined to be $0$.


It turns out that $X_0',X_1',X_2',\dots$ evolves in a manner very similar to the evolution of $X_0,X_1,\dots$. Indeed, relying on the method of the proof of Lemma~\ref{lemmaMarkov} (in particular Remark~\ref{remarkI}), one can see that $X_0',X_1',X_2',\dots$ is a time inhomogeneous Markov chain, where at every second step we have the same transition probabilities as the Markov chain $X_0,X_1,X_2,\dots$, that is
\[\mathbb{P}(X_{2i+1}'=r\,|\,X_{2i}'=d)=\mathbb{P}(X_{2i+1}=r\,|\,X_{2i}=d).\]
However, the transition from $X_{2i-1}'$ to $X_{2i}'$ is completely deterministic, namely, 
\[X_{2i}'=n-X_{2i-1}'.\]

Note that
\begin{equation}\label{dimker'}
\dim\ker C_{2k-1}'=kn-\sum_{i=1}^{k-1}(X_{2i-1}'+X_{2i}')-X_{2k-1}'=n-X_{2k-1}'. 
\end{equation}

Moreover, if $X_i$ is defined as in \eqref{rankincrementdef}, then
\begin{align}\dim\ker C_{2k-2}&=\sum_{i=1}^{k-1}(n-X_{2i-1}-X_{2i})\label{dimkerC2k-2}\\
\dim\ker C_{2k-1}&=\dim\ker C_{2k-2}+n-X_{2k-1}\label{dimkerC2k-1eq}.
\end{align}

\begin{lemma}\label{almostdeter}
 Assume that $\lim_{n\to\infty}\frac{q^{-n/2}}{q-1}k_n=0$. Then, as $n\to\infty$ with probability tending to one, we have
 \[X_{2i}=n-X_{2i-1}\text{ for all }1\le i\le k_n-1\text{ and }\dim\ker C_{2k_n-1}=n-X_{2k_n-1}.\]
\end{lemma}
\begin{proof}
 From part~\eqref{thmdimkerC2kpart1} of Theorem~\ref{thmdimkerC2k}, it follows that
 \[\lim_{n\to\infty}\mathbb{P}(\dim\ker C_{2k_n-2}^{(n)}=0)=1.\]
 
 Note that $n-X_{2i-1}-X_{2i}\ge 0$ for all $i$. Thus, by \eqref{dimkerC2k-2} and \eqref{dimkerC2k-1eq}, $\dim\ker C_{2k_n-2}=0$ implies that $X_{2i}=n-X_{2i-1}$ for all $1\le i\le k_n-1$ and $\dim\ker C_{2k_n-1}=n-X_{2k_n-1}$. 
\end{proof}

Part~\eqref{thmdimkerC2k-1part1} of Theorem~\ref{thmdimkerC2k-1} follows by combining \eqref{productasbidiagonal}, \eqref{dimker'}, Lemma~\ref{almostdeter} and the following lemma.

\begin{lemma}
Assume that $\lim_{n\to\infty}\frac{q^{-n/2}}{q-1}k_n=0$. Then there is a coupling of \break $(X_0,X_1,\dots,X_{2k_n-1})$ and  $(X_0',X_1',\dots,X_{2k_n-1}')$ such that with probability tending to $1$, we have $(X_0,X_1,\dots,X_{2k_n-1})=(X_0',X_1',\dots,X_{2k_n-1}')$.
\end{lemma}
\begin{proof}
On the event that $X_{2i}=n-X_{2i-1}$ for all $1\le i\le k_n-1$, we set $X_i'=X_i$ for all $0\le i\le 2k_n-1$. Otherwise, let $\tau$ be the smallest $i$ such that $X_{2i}\neq n-X_{2i-1}$. Then we set $X_i'=X_i$ for all $0\le i\le 2\tau-1$. And for $i\ge 2\tau$, we let $X_i'$ evolve according to the transition probabilities given for $X'$. It is straightforward to see that this is a coupling of $(X_0,X_1,\dots,X_{2k_n-1})$ and $(X_0',X_1',\dots,X_{2k_n-1}')$. Moreover, by Lemma~\ref{almostdeter}, we have
\[\lim_{n\to\infty}\mathbb{P}((X_0,X_1,\dots,X_{2k_n-1})=(X_0',X_1',\dots,X_{2k_n-1}'))=1.\qedhere\]
\end{proof}







\begin{remark}\label{remarkdeviateafterwhile}
    If $\lim_{n\to\infty} \frac{q^{-n/2}}{q-1}k_n=t>0$, then asymptotic distribution of the number of indices $1\le i\le k_n-1$ such that $X_{2i}\neq n-X_{2i-1}$ is given by $D_{-\infty,t}$. In particular, the probability that $X_{2i}\neq n-X_{2i-1}$ for some $1\le i\le k_n-1$ is bounded away from zero. Thus, Lemma~\ref{almostdeter} cannot be extended to this case. In fact, $\dim \ker C_{2k_n-1}$ and $\dim \ker C_{2k_n-1}'$ show a different behavior. Since $t^{-1}(D_{-\infty,t}+Z_{-\infty,t})$ converges to some positive constant almost surely, one can see that $\dim \ker C_{2k_n-1}-\frac{n}2$ is concentrated around a constant $c_t$, where $c_t$ grows linearly with $t$. On the other hand, $\dim \ker C_{2k_n-1}'-\frac{n}2$ is concentrated around $\log_q(t)$ by the results of Van Peski~\cite{van2023local}. Thus, after reaching the critical value of $k_n\approx q^{n/2}$, $\dim \ker C_{2k_n-1}$ grows faster than $\dim \ker C_{2k_n-1}'$.
\end{remark}

\subsection{The rest of the proof of Theorem~\ref{thmdimkerC2k-1}}\label{secrestof}

\begin{lemma}\label{lemmacrossing}
Let $b<a$. Then
\begin{align*}\mathbb{P}(X_{2i+2}\ge a \,|\,X_{2i}=b)&=O(q^{-a}),\text{ and }\\
\mathbb{P}(X_{2i+2}\le b \,|\,X_{2i}=a)&=O(q^{-(n-b)}).
\end{align*}
\end{lemma}
\begin{proof}
Note that $n-X_{2i+1}\ge X_{2i+2}$. Thus, the event that $X_{2i+2}\ge a$ is contained in the event that $n-X_{2i+1}\ge a$. Thus,
\begin{multline*}\mathbb{P}(X_{2i+2}\ge a \,|\,X_{2i}=b)\le \mathbb{P}(X_{2i+1}\le n-a \,|\,X_{2i}=b)\\=\sum_{r=0}^{n-a} P_n(b,r)=\sum_{r=0}^{n-a} O(q^{-(n-b-r)(n-r)})=O(q^{-(a-b)a})=O(q^{-a}).\end{multline*}
This gives the first statement, the second follows similarly.
\end{proof}

\begin{lemma}\label{tight}
For any $\varepsilon>0$, there are constants $0<t,K<\infty$ such that for all large enough $n$, any initial value of $X_0=X_{0}^{(n)}$, if $j$ is such that $j\ge (q-1)q^m t$, then
\[\mathbb{P}(|X_{2j}-m|\ge K)\le \varepsilon.\]

\end{lemma}
\begin{proof}
Let $\tau=\min\{i\ge j-(q-1)q^m t\,:\,X_{2i}=m\}$.

By Lemma~\ref{Utail}, if $t$ is large enough, then for all large enough $n$, we have
\[\mathbb{P}(\tau\le j)\ge 1-\frac{\varepsilon}2.\]
By Lemma~\ref{lemmacrossing}
\[\mathbb{P}\left(|X_{2i+2}-m|\ge K\,\Big|\,|X_{2i}-m|<K\right)\le O(q^{-m-K}).\]

Thus,
\begin{align*}
\mathbb{P}(|X_{2j}-m|< K \,|\,\tau\le j)&\ge (1-O(q^{-m-K}))^{j-\tau}\\&\ge 1-tq^m(q-1)O(q^{-m-K})\\&=1-O(tq^{-K})\\&\ge 1-\frac{\varepsilon}2\end{align*}
provided that $K$ is large enough.
\end{proof}

\begin{lemma}\label{tightanddimkerC2k1}
 Let $k_n$ be a sequence of positive integers, and let $t_n=\frac{q^{-m}}{q-1} k_n$. Assume that $\lim_{n\to\infty} t_n=t$, where $t>0$ (including the possibility of $t=\infty$). Then 
 \begin{enumerate}[(a)]
  \item Then the sequence of the random variables $X^{(n)}_{0,2k_n}-m$ is tight.
  \item \[\lim_{n\to\infty}\mathbb{P}(X^{(n)}_{0,2k_n+1}=n-X^{(n)}_{0,2k_n})=1.\]
  \item The total variation distance of $\dim\ker C_{2k_n+1}^{(n)}$ and $D'_{0,k_n,n}+(X_{0,2k_n}^{(n)}-m)+m$ converges to $0$.
 \end{enumerate}
\end{lemma}
\begin{proof}
Part (a): If $t<\infty$, this follows from the fact that by Lemma~\ref{criticalconvergence}$, X^{(n)}_{0,2k_n}-m$ converges in total variation distance to $Z_{-\infty,t}$. For $t=\infty$, this follows from Lemma~\ref{tight}.

Part (b): Note that for any fixed $a$, we have
\[\lim_{n\to \infty} P_n(m+a,m-a)=1.\]
The statement follows by combining this with part (a).

Part (c): We have
\[\dim \ker C_{2k_n+1}=\dim \ker C_{2k_n}+n-X_{0,2k_n+1}=D'_{0,k_n,n}+n-X_{0,2k_n+1}^{(n)}.\]
Thus, on the high probability event $X^{(n)}_{0,2k_n+1}=n-X^{(n)}_{0,2k_n}$
\[\dim \ker C_{2k_n+1}=D'_{0,k_n,n}+(X_{0,2k_n}^{(n)}-m)+m.\qedhere\]
\end{proof}

Part~\eqref{thmdimkerC2k-1part2} of Theorem~\ref{thmdimkerC2k-1} follows by combining part (c) of Lemma~\ref{tightanddimkerC2k1} and Lemma~\ref{criticalconvergence}.

To prove part \eqref{thmdimkerC2k-1part3} of Theorem~\ref{thmdimkerC2k-1}, recall that in \eqref{GaussasD}, we proved that
\[\frac{D'_{0,k_n}-\mu_n t_n}{\sigma\sqrt{t_n}}\]
converges in distribution to a standard normal variable.

By part (a) of Lemma~\ref{tightanddimkerC2k1},
\[\frac{X_{0,2k_n}^{(n)}-m}{\sigma \sqrt{t_n}}\]
converges to $0$ in probability.

Thus,
\[\frac{D'_{0,k_n}+X_{0,2k_n}^{(n)}-m-\mu_n t_n}{\sigma\sqrt{t_n}}\]
converges in distribution to a standard normal variable. Combining this with part (c) of Lemma~\ref{tightanddimkerC2k1}, part~\eqref{thmdimkerC2k-1part3} of Theorem~\ref{thmdimkerC2k-1} follows.

\section{The case of odd $n$}\label{oddn}
Throughout this section, we assume that $n$ is odd.

Let $\left(Z_t^{\text{odd}}\right)$ be a continuous time random walk on \[\mathbb{Z}+\frac{1}2=\left\{a+\frac{1}2\,:\,a\in\mathbb{Z}\right\}\]
with transition rate matrix
\[Q^{\text{odd}}_Z(a,b)=\begin{cases}
q^{-a}&\text{if }b=a+1,\\
q^a&\text{if }b=a-1,\\
-q^a-q^{-a}&\text{if }$a=b$,\\
0&\text{otherwise.}
\end{cases}\]

The initial distribution of $Z_0^{\text{odd}}$ is to be specified later.

Let $T_i, X_i,Y_i,R_i$ be defined as before but using $Z^{\text{odd}}$ in place of $Z$ and assuming $n$ is odd instead of assuming that $n$ is even.

\begin{theorem}\label{theoremcouplingodd}
Let $\alpha=0.08$. Assume that $Y^{(n)}_0=Z_0^{\text{odd}}=a$, where $a$ is a deterministic constant satisfying $|a|\le 2\alpha n$. Then there is a coupling of the processes $Z_t^{\text{odd}}$ and $Y_t=Y_t^{(n)}$ such that with probability at least $1-O\left(q^{-2\alpha n}\right)$ the following holds:
\begin{enumerate}[(a)]
\item For all $0\le i\le q^{2\alpha n}$, we have $\widehat{Y}_i=\widehat{Z}^{\text{odd}}_i$;
\item For all $0\le i\le q^{2\alpha n}$, we have $|R_i-T_i|\le q^{-\alpha n}$;
\item $T_{\lfloor q^{2\alpha n} \rfloor}\ge q^{\alpha n}$.
\end{enumerate}
\end{theorem}

\begin{corollary}\label{marginalconvodd}
 Let $0<t_1<t_2<\cdots<t_k$ be a deterministic sequence of times. Assume that ${Z}^{\text{odd}}_0$ is constant, and assume that for all large enough $n$, we have $Y_0^{(n)}=Z_0^{\text{odd}}$. Then
$\left(Y_{t_1}^{(n)},Y_{t_2}^{(n)},\dots,Y_{t_k}^{(n)}\right)$ converges to $\left(Z^{\text{odd}}_{t_1},Z^{\text{odd}}_{t_2},\dots,Z^{\text{odd}}_{t_k}\right)$ in total variation distance.
\end{corollary}

For $a\in \mathbb{Z}+\frac{1}2$ and $t>0$, let $D^{\text{odd}}_{a,t}$ be the number of downward jumps made by $Z^{\text{odd}}$ until time~$t$ under the initial condition $Z_0^{\text{odd}}=a$.

\begin{lemma}\label{Dinftyexistsodd}
For any fixed $t>0$, as $a$ tends to $-\infty$, $(D_{a,t}^{\text{odd}},Z_{a,t}^{\text{odd}})$ converge in total variation distance to some random variable $(D_{-\infty,t}^{\text{odd}},Z_{-\infty,t}^{\text{odd}})$. 
\end{lemma}

\begin{theorem}\label{thmdimkerC2kodd}

Let $k_n$ be a sequence of positive integers, and let $t_n=\frac{q^{-n/2}}{q-1}k_n$. 
 \begin{enumerate}[(1)]
\item Assume that $\lim_{n\to\infty}t_n=0$, then $\dim\ker {C}^{(n)}_{2k_n}$ converges to $0$ in probability.
\item Assume that $\lim_{n\to\infty}t_n=t$, where $0<t<\infty$, then $\dim\ker {C}^{(n)}_{2k_n}$ converges to $D^{\text{odd}}_{-\infty,t}$ in total variation distance.

\item Let $\mu_n$ be defined as in \eqref{mundef}. There is a constant $\sigma^{\text{odd}}>0$ such that if $\lim_{n\to\infty}t_n=\infty$, then
\[\frac{\dim\ker {C}^{(n)}_{2k_n}-\mu_n t_n}{\sigma^{\text{odd}} \sqrt{t_n}}\]
converges in distribution to a standard normal variable.
Moreover,
\begin{equation}\label{munas2}\mu_n=\frac{\sqrt{q}\left({1}+2\sum_{i=1}^{\infty} q^{-i^2}\right)}{2\sum_{i=1}^\infty q^{-i(i-1)}}+O(q^{-n/2}).\end{equation}
\end{enumerate}
\end{theorem}

Theorem~\ref{thmdimkerC2k-1} has a similar analogue for odd $n$.

As the proofs in the case of odd $n$ are almost the same as in the case of even $n$, we only provide proofs for the even case.

Note that although $\mu_n$ is defined by the same formula for both even and odd $n$, the asymptotic value of $\mu_n$ depends on the parity of $n$. The proof of~\eqref{munas2} follows along the lines of the proof of~\eqref{munas} given in Section~\ref{SecMunas}. However, since the reader might wonder why the two formulas look so different, we provide a sketch of the proof.

In this proof, we set $m=\frac{n-1}2$.

For $i\ge 1$, we have
\begin{align*}
\frac{\pi_n(m+i)}{\pi_n(m)}&=q^{i-i^2}\prod_{j=0}^{i-1}(1-q^{-m+j-1})^2(1-q^{-m-j-1})^{-1}\\&=
q^{i-i^2}\left(1-\sum_{j=0}^{i-1}\left(2q^{-m+j-1}-q^{-m-j-1}\right)+O\left(q^{-2(m-i)}\right)\right).
\end{align*}

A similar argument gives that for $i\ge 1$, we have
\[\frac{\pi_n(m+1-i)}{\pi_n(m)}=q^{i-i^2}\left(1-\sum_{j=0}^{i-2}\left(q^{-m+j}-2q^{-m-j-2}\right)+O\left(q^{-2(m-i)}\right)\right),\]
and so
\begin{align*}\frac{\pi_n(m+1-i)-\pi_n(m+i)}{\pi_n(m)}&=q^{i-i^2-m}\left(\frac{q^{i-1}-q^{-i}}{q-1}\right)+O\left(q^{-i^2+3i-n}\right).
\end{align*}

By \eqref{mundef}, we have
\begin{align*}
\frac{\mu_n}{\pi_n(m)}&=\frac{(q-1)q^m\sqrt{q}}{\pi_n(m)} \left(n-2\sum_{i=0}^{n} \pi_n(i)i\right)\\&=2(q-1)q^m\sqrt{q}\sum_{i=1}^{m+1} \left(\frac{\pi_n(m+1-i)-\pi_n(m+i)}{\pi_n(m)}\right)\left(i-\frac{1}2\right)\\&
=2(q-1)q^m\sqrt{q}\sum_{i=1}^{m+1} \left(q^{i-i^2-m}\frac{q^{i-1}-q^{-i}}{q-1}+O(q^{-i^2+3i-n})\right)\left(i-\frac{1}2\right)\\&=
2\sqrt{q}\sum_{i=1}^{m+1} \left(q^{-(i-1)^2}-q^{-i^2}\right)\left(i-\frac{1}2\right) +O(q^{-m})\\&
=\sqrt{q}\left(1+2\sum_{i=1}^\infty q^{-i^2}\right)+O(q^{-m}).
\end{align*}

The above argument also gives that for $i\ge 1$, we have
\[\frac{\pi_n(m+1-i)+\pi_n(m+i)}{\pi_n(m)}=2q^{i-i^2}+O(q^{-i^2-m+2i}).\]

Thus,
\begin{align*}\frac{1}{\pi_n(m)}&=\sum_{i=1}^{m+1}\frac{\pi_n(m+1-i)+\pi_n(m+i)}{\pi_n(m)}\\&=2\sum_{i=1}^\infty q^{-i(i-1)} +O(q^{-m}).\end{align*}

Therefore,
\[\mu_n=\frac{\sqrt{q}\left({1}+2\sum_{i=1}^{\infty} q^{-i^2}\right)}{2\sum_{i=1}^\infty q^{-i(i-1)}}+O(q^{-n/2}).\]

\bibliography{references}
\bibliographystyle{plain}

\noindent\textbf{Acknowledgments.} The author is grateful to the
anonymous referee for their comments and useful suggestions.
\end{document}